\documentclass[12pt,reqno]{amsart}
\usepackage[utf8]{inputenc}
\usepackage{colortbl}
\usepackage{hhline}
\usepackage{tikz}
\usepackage{float}
\usepackage{amsmath,amsthm,amsfonts,amssymb}
\usepackage{color}
\usepackage{cases}
\usepackage{wasysym}
\usepackage{empheq}
\usepackage{graphicx}
\usepackage[english]{babel}
\usepackage[numbers]{natbib}
\usepackage{ytableau}
\usepackage{yfonts}

\numberwithin{equation}{section}

\newtheorem{theorem}{Theorem}
\numberwithin{theorem}{section}
\newtheorem{proposition}[theorem]{Proposition}
\newtheorem{lemma}[theorem]{Lemma}
\newtheorem{corollary}[theorem]{Corollary}

\theoremstyle{definition}
\newtheorem{definition}[theorem]{Definition}

\theoremstyle{remark}
\newtheorem*{remark}{Remark}

\newtheorem{example}[theorem]{Example}

\newcommand{\noteB}[1]{ { \sffamily \color{red} \normalsize #1} }

\def\lcm{\operatorname{lcm}}

\date{\today}
\usepackage{subfiles}
\title{An Orthogonal View of Gau\ss ian Polynomials}
\author{Christian Krattenthaler}
\address{Fakult\"at f\"ur Mathematik,
Universit\"at Wien,
Oskar-Morgenstern-Platz 1,
A-1090 Vienna AUSTRIA}
\email{Christian.Krattenthaler@univie.ac.at}
\author{Brandt Kronholm}
\address{School of Mathematical and Statistical Sciences,
  University of Texas Rio Grande Valley,
  Edinburg, Texas 78539-2999 USA}
\email{brandt.kronholm@utrgv.edu}
\author{Paul Marsh}
\address{School of Mathematical and Statistical Sciences,
  University of Texas Rio Grande Valley,
  Edinburg, Texas 78539-2999 USA}
\email{paul.marsh01@utrgv.edu}

\thanks{The first author was partially supported by the Austrian
Science Foundation FWF, grant 10.55776/F1002, in the framework
of the Special Research Program ``Discrete Random Structures:
Enumeration and Scaling Limits"}

  \renewcommand{\geq}{\geqslant}
  \newcommand{\N}[0]{\mathbb{N}}		
  \newcommand{\Z}[0]{\mathbb{Z}}		
  \newcommand{\C}[0]{\mathbb{C}}		

\newcommand{\m}[0]{\mathcal{M}}

\def\om{\omega}
\def\fl#1{\lfloor#1\rfloor}

\def\coef#1{\left\langle#1\right\rangle}
\def\qbin#1#2{\left[\begin{smallmatrix} #1\\#2\end{smallmatrix}\right]}
\def\Qbin#1#2{\begin{bmatrix} #1\\#2\end{bmatrix}}

\begin{document}

\keywords{Integer partition, Gau\ss ian polynomial, generating function, partition identity, unimodality}
  
\subjclass[2010]{Primary 11P81; Secondary 05A17, 05A15, 05A19}

\begin{abstract}
We establish an alternative, ``perpendicular" collection of generating\break functions for the coefficients of Gau\ss ian polynomials,
$\qbin{N+m}{ m}_q$.  
We provide a general characterization of these perpendicular generating functions.  
For small values of~$m$, unimodality of the coefficients of Gau\ss ian polynomials is easily proved from these generating functions.  
Additionally, we uncover new and surprising identities for the differences of Gau\ss ian polynomial coefficients, including a very unexpected infinite family of congruences for coefficients of $\qbin{N+4}{ 4}_q$.  
\end{abstract}
 
\maketitle

\section{Introduction}\label{Intro}

In this paper we establish an {\it alternative} collection of generating functions for the coefficients of Gau\ss ian polynomials.  
While it may be unexpected that a completely new set of generating functions should exist, they come about by making use of an overlooked technique in partitions noted independently by H.~Gupta~\cite{GuptaTechnique,gupta} in 1975, but known and well studied about a decade earlier by E.~Ehrhart~\cite{BEK-6j-1,ehrhartpolynomial} in the area of polyhedral geometry.  
An analysis of these alternative generating functions --- {\it perpendicular}, as we shall call them
frequently due to their nature of collecting coefficients --- has the
happy by-product of new proofs for the unimodality of the coefficients
of Gau\ss ian polynomials $\qbin{N+m}m_q$ for small values of~$m$.  
Following this, we establish a collection of surprising difference identities for partitions with bounded largest part and bounded number of parts.  

\subsection{Standard generating functions for Gau\ss ian polynomials}

Before we go any further, we define the {\it Gau\ss ian polynomial},
also known as the {\it $q$-binomial coefficient}.  

\begin{definition}\label{GP}
For $m,N\ge0$ the expression below is known as a Gau\ss ian polynomial or a $q$-binomial coefficient,  
\begin{equation}\label{ratfun}
\Qbin{N+m}{ m}_q
=\frac 
{(q;q)_{N+m}}{(q;q)_m(q;q)_N}
= \frac{\left(q^{N+1};q\right)_{m}}{(q;q)_m},
\quad \text{for } m,N\ge0,
 \end{equation}
where $(z;q)_a:=(1-z)(1-zq)(1-zq^2)\cdots(1-zq^{a-1})$ if $a$ is a
positive integer, and $(z;q)_0:=1$. 
\end{definition}

The coefficients of Gau\ss ian polynomials have a well-known interpretation in terms of integer partitions.  

\begin{definition}
    A {\it partition} of a positive integer $n$ is a finite nonincreasing sequence of positive integers $\lambda_1,\lambda_2,...,\lambda_r$ such that $\sum^r_{i=1} \lambda_i =n$. 
  The $\lambda_i$ are called the parts of the partition. 
  
In this paper we will make use of the following two partition functions:
\begin{itemize}
    \item $p(n,m)$: enumerates the partitions of $n$ into at most $m$ parts, and
    \item $p(n,m,N)$: enumerates the partitions of $n$ into at most $m$ parts with no part larger than~$N$.  
\end{itemize}
\end{definition} 
\begin{proposition}\label{GP=p(n,m,N)}
For $n,m,N\ge0$, the Gau\ss ian polynomial $\qbin{N+m}{ m}_q$ is the generating function for 
$p\big(n,m,N\big)$; that is,
\begin{equation}
    \Qbin{N+m}{ m}_q=\sum_{n=0}^{mN}p\big(n,m,N\big)q^n.
\end{equation}
\end{proposition}
A proof of Proposition~\ref{GP=p(n,m,N)} can be found
in~\cite[Theorem~3.1]{andrews1998theory}.  
Clearly, for $0\le n\le mN$,  $p\big(n,m,N\big)>0$, otherwise, $p\big(n,m,N\big)=0$.  
Hence, $\qbin{N+m}{ m}_q$ is a polynomial of degree $mN$ with $mN+1$ terms.

\begin{example}\label{4}
    For a given $N$, a Gau\ss ian polynomial $\qbin{N+4}{4}_q$ is computed by expanding the following rational function and arriving at the associated generating function for $p(n,4,N)$:
    \begin{equation}
        \Qbin{N+4}{4}_q = \frac{(1-q^{N+1})(1-q^{N+2})(1-q^{N+3})(1-q^{N+4})}{(1-q)(1-q^2)(1-q^3)(1-q^4)} = \sum_{n=0}^{4N}p(n,4,N)q^n.
    \end{equation}
\end{example}

In this paper we establish entirely new ``perpendicular'' generating functions for $\qbin{N+m}{ m}_q$.  
For example, for a given $N$, expansion of the generating function below in Proposition~\ref{C_4a} recovers the Gau\ss ian polynomial $\qbin{N+4}{ 4}_q$ by collecting the terms as $-2N\le A\le 2N$. 

\begin{proposition}\label{C_4a}
For all $A\ge0$, we have
\begin{equation*}
  \sum_{N=0}^{\infty}p\big(2N-A,4,N\big)z^N =\begin{cases}
\displaystyle
\frac{z^{a}\left(1+z^2-z^{a+1}\right)}{\left(1-z\right)^2\left(1-z^2\right)\left(1-z^3\right)},&
\text{if }A=2a,\\
\displaystyle
\frac{z^{a}\left(z+z^2-z^{a+2}\right)}{\left(1-z\right)^2\left(1-z^2\right)\left(1-z^3\right)},&
\text{if }A=2a+1.
      \end{cases}
\end{equation*}
    
\end{proposition}

The alternative collection of generating functions that we consider produces the coefficients of $\qbin{N+m}{ m}_q$ for all $N$ and a fixed $m$ depending on how far the coefficient is from the {\em center} of the Gau\ss ian polynomial.  
For all $N$, the generating function for $p(n,4,N)$ in Example~\ref{4} generates the same coefficients as that of Proposition~\ref{C_4a} for all $A$.
(Proposition~\ref{C_4a} will be restated as Proposition~\ref{C_4b} later and proved there.)

With modest computing power, we have obtained perpendicular generating functions for $m=1,2,\dots,12$.   
Our methods extend to all $m\in\N$.

\subsection{Background material}
To produce this alternative collection of generating functions for Gau\ss ian polynomials, we review some well-known facts and establish a few definitions.  

It is well known that Gau\ss ian polynomials are reciprocal polynomials.    
In other words, the coefficients of a Gau\ss ian polynomial form a palindrome. 

\begin{definition}\label{reciprocal}
A polynomial $P(q)=a_0+a_1q+a_2q^2+\cdots +a_dq^d$ is called {\em reciprocal} if for each $i$, $a_i=a_{d-i}$, equivalently, if $q^dP(q^{-1})=P(q)$.   
\end{definition}

\begin{example}\label{ex33}
Two Gau\ss ian polynomials:
    \begin{multline}\label{ex33eq}
        \Qbin{3+3}{ 3}_q=\frac{(q;q)_{6}}{(q;q)_3(q;q)_3}=1+q+2q^2+3q^3+3q^4+3q^5+3q^6+2q^7+q^8+q^9
        \\=\sum_{m=0}^{9}p\left(n,3,3\right)q^n.
    \end{multline}
\begin{multline}\label{ex43eq}
    \Qbin{3+4}{ 4}_q=\frac{\left(q;q \right)_7}{\left(q;q \right)_4\left(q;q \right)_3}
    \\ = 1+q+2q^2+3q^3+4q^4+4q^5+5q^6+4q^7+4q^8+3q^9+2q^{10}+q^{11}+q^{12}
    \\=\sum_{m=0}^{12}p\left(n,4,3\right)q^n.
\end{multline}
Noting that the coefficient of $q^{10}$ is 2, we see that there are two partitions of~$10$ into at most four parts with no part larger than $3$ and we write $p(10,4,3)=2$.  
The relevant partitions are $3+3+3+1$ and $3+3+2+2$.  
Since Gau\ss ian polynomials are reciprocal, we also have $p(2,4,3)=2$, and the relevant partitions are $2$ and $1+1$. 
\end{example}

Gau\ss ian polynomials have one, sometimes two, ``middle" or {\em central}
terms.  
Since we require an unambiguous single coefficient to be our {\em central} coefficient, we provide a definition. 
\begin{definition}\label{centralcoeff}
    We define $p\left(\left\lfloor\frac{mN}{2}\right\rfloor,m,N \right)$ to be the {\em central} coefficient of $\qbin{N+m}{ m}_q$.
\end{definition}

\begin{example}\label{centraldef}
    In line \eqref{ex33eq} of Example~\ref{ex33}, there are exactly two coefficients in the middle of $\qbin{3+3}{ 3}_q$: $p\left(4,3,3 \right)\mbox{ and }p\left(5,3,3 \right).$   
Adhering to Definition \ref{centralcoeff}, we select the term $p\left(\left\lfloor\frac{3\times3}{2}\right\rfloor,3,3 \right)q^{\left\lfloor\frac{3\times3}{2}\right\rfloor}=p(4,3,3)q^4$, so that $p(4,3,3)=3$ is the central coefficient in this case.  
In line~\eqref{ex43eq} of Example~\ref{ex33}, there is a single middle term, and so the central coefficient is $p\left(\left\lfloor\frac{3\times4}{2}\right\rfloor,3,4 \right)=p\big(6,3,4\big)=5$. 
\end{example}  

\begin{remark}\label{A} 
Let $A$ be an integer.  
Since $\qbin{N+m}{ m}_q$ is reciprocal, we note that for $m\times N$ even, \begin{equation}\label{palindromeE}
p\left(\left\lfloor\frac{mN}{2}\right\rfloor-A,m,N \right)=p\left(\left\lfloor\frac{mN}{2}\right\rfloor+A,m,N \right), 
\end{equation}
and for $m\times N$ odd,
\begin{equation}\label{palindromeO}
p\left(\left\lfloor\frac{mN}{2}\right\rfloor-A,m,N \right)=p\left(\left\lfloor\frac{mN}{2}\right\rfloor+A+1,m,N \right).
\end{equation}
\end{remark}

\begin{example}\label{central-Aex}
    In line \eqref{ex33eq} of Example~\ref{ex33}, we again examine the two coefficients in the middle of $\qbin{3+3}{ 3}_q$.  
    By Remark \ref{A}, for $A=0$ we obtain the central coefficient as\break
    $p\left(\left\lfloor\frac{3\times3}{2}\right\rfloor-0,3,3 \right) = p\left(4,3,3 \right)$.  
    Since $3\times3$ is odd we have 
    $p\left(4,3,3 \right)= p\left(5,3,3 \right)$ by~\eqref{palindromeO}.   
\end{example}

Remark \ref{A} is our starting point for creating these alternative
--- perpendicular --- generating functions.  
Table \ref{p(n,4,N)TABLE} displays the first eight polynomials $\qbin{N+4}{ 4}_q$, for $0\le N\le 7$, ``stacked" around the central coefficient $p(2N,4,N)$.  
The generating functions that we produce are not for a single Gau\ss ian polynomial $\qbin{N+m}{ m}_q$ for given a pair $m$ and $N$, but rather describe the sequence of coefficients $p\left(\left\lfloor\frac{mN}{2}\right\rfloor-A,m,N \right)$ of all Gau\ss ian polynomials for a given pair $m$ and $A$ {\em for all} $N$.  
In this light we say that the alternative generating functions we produce are ``perpendicular" to the standard generating functions.

\begin{table}[H]
\small
{\begin{center}
\begin{tabular}{@{}cc@{}}
${\qbin{0+4}{ 4}}_q=$    & \!\!\!\!\!\!\!\!$1$\\[2pt]
${\qbin{1+4}{ 4}}_q=$   & $1~+~q~+~q^2+~q^3~+~q^4$ \\[2pt]
${\qbin{2+4}{ 4}}_q=$& $1+q+2 q^2+2 q^3+3 q^4+2 q^5+2 q^6+q^7+q^8$ \\[2pt]
${\qbin{3+4}{ 4}}_q=$&  $~~~1+q+2 q^2+3 q^3+4 q^4+4 q^5+5 q^6+4 q^7+4 q^8+3 q^9+2 q^{10}+q^{11}+q^{12}$\\[2pt]
${\qbin{4+4}{ 4}}_q=$&  $~\cdots+3q^4 +5 q^4+5 q^5+7 q^6+7 q^7+8 q^8+7 q^9+7q^{10}+5q^{11}+5 q^{12}+3q^{13}+\cdots$\\[2pt]
${\qbin{5+4}{ 4}}_q=$&  $~~~\cdots+8 q^6+9 q^7+11 q^8+11 q^9+12 q^{10}+11 q^{11}+11 q^{12}+9 q^{13}+8 q^{14}+\cdots$\\[2pt]
${\qbin{6+4}{ 4}}_q=$&$\cdots+13 q^8+14 q^9+16 q^{10}+16 q^{11}+18 q^{12}+16 q^{13}+16 q^{14}+14 q^{15}+13 q^{16}+\cdots$\\
${\qbin{7+4}{ 4}}_q=$&
$\cdots+19 q^{10}+20 q^{11}\!+23 q^{12}\!+23 q^{13}\!+24 q^{14}+23 q^{15}+23 q^{16}+20 q^{17}+19 q^{18}+\cdots$\\
\hspace{12pt}\vdots \hspace{12pt}=& \vdots~~~
\end{tabular}
\end{center}}
\caption{The sequence of Gau\ss ian polynomials 
  ${\protect\qbin{N+4}{ 4}}_q$
  arranged with respect to their central coefficients.  
   The sequence of central coefficients is $\{1,1,3,5,8,12,18,24,\ldots\}$ which is reflected in the generating function in Example~\ref{p(n,4,N)ex} where $A=0$.  
    Similarly, the sequence of coefficients ``one-away" from the central coefficient is $\{0,1,2,4,7,11,16,23,\ldots\}$ and corresponds to the generating function in Example~\ref{p(n-1,4,N)ex} where $A=1$.
}
    \label{p(n,4,N)TABLE}
\end{table}

The possibilities of this area of investigation were indicated in~\cite{LSAMPpn3N} and, in several ways, this article is an overdue followup of~\cite{LSAMPpn3N}.  
Example \ref{p(n,4,N)ex}, below, was initially established in Equation (4.27) in~\cite{LSAMPpn3N}.  
\begin{example}[\cite{LSAMPpn3N}]\label{p(n,4,N)ex}
For any $N$, the central coefficient of $\qbin{N+4}{ 4}_q$ is $p(2N,4,N)$.  
The generating function for $p(2N,4,N)$ is
\begin{multline}\label{p(2N,4,N)exeq}
    \sum_{N=0}^{\infty}p\big(2N,4,N\big)z^N = \frac{1-z+z^2}{\left(1-z\right)^2\left(1-z^2\right)\left(1-z^3\right)} = 1+z+3 z^2+5 z^3+8 z^4\\+12 z^5+18 z^6+24 z^7 +33 z^8+43 z^9+55 z^{10}+69 z^{11}+86 z^{12}+104 z^{13}+126 z^{14}+ \cdots.
\end{multline}
\end{example}
Compare the coefficients in \eqref{p(2N,4,N)exeq} to the sequence of central coefficients in Table \ref{p(n,4,N)TABLE}.

Example \ref{p(n-1,4,N)ex}, below, is the generating function for the coefficients that ``precede" the central coefficient", or better, $A=1$, of Gau\ss ian polynomials $\qbin{N+4}{ 4}_q$ and is new.   

\begin{example}\label{p(n-1,4,N)ex}
For any $N$, the generating function for $p(2N-1,4,N)$ is:
\begin{multline}\label{p(2N-1,4,N)exeq}
    \sum_{N=0}^{\infty}p\big(2N-1,4,N\big)z^N = \frac{z}{\left(1-z\right)^2\left(1-z^2\right)\left(1-z^3\right)} = z+2 z^2+4 z^3+7 z^4\\+11 z^5+16 z^6+23 z^7+31 z^8+41 z^9+53 z^{10}+67 z^{11}+83 z^{12}+102 z^{13}+123 z^{14}+\cdots.
\end{multline}
\end{example}
Again, compare the coefficients in \eqref{p(2N-1,4,N)exeq} to the sequence of coefficients immediately to the left of the central coefficients in Table \ref{p(n,4,N)TABLE}.

By setting $a=0$, Example \ref{p(n,4,N)ex} and Example~\ref{p(n-1,4,N)ex} are extracted from the perpendicular generating function for $\qbin{N+4}{ 4}_q$ in Proposition~\ref{C_4a}.  

\subsection{How this paper is structured}\label{sectionstructure}

In Section \ref{establishing} we present our main results for our
perpendicular partition generating functions, separately for even~$m$
and for odd~$m$; see Theorems~\ref{thm:even} and~\ref{thm:odd}.
We illustrate these general results by displaying the corresponding
results for $m=1,2,\dots,6$, which we obtained with the implementation
of the results in the accompanying
{\sl Mathematica} Notebook {\tt orthview.nb}.

After the procedure is established, we follow up with short proofs of
unimodality in Section~\ref{uni}.   
In Section~\ref{ID} we prove many unexpected identities for the differences of Gau\ss ian polynomial coefficients for $N=3,4,5,6$.  
Included in these observations is a very short proof of Proposition~\ref{4isD3}; line~\eqref{repp(2N-1-2,4,N)eq2} is of interest to Lie Algebraists.

\begin{proposition}\label{4isD3}
Let $N$ be any nonnegative integer.  
Then
\begin{gather}\label{repp(2N-1-2,4,N)eq1}
p\big(2N,4,N \big)-p\big(2N-1,4,N \big)=
p\big(N,3\big)-p\big(N-1,3\big),\\
\label{repp(2N-1-2,4,N)eq2}        
p(2N-1,4,N)-p(2N-2,4,N) = 0
\end{gather}
\end{proposition}

Line \eqref{repp(2N-1-2,4,N)eq1} can be read as {\em the difference between the largest and second largest coefficient of any Gau\ss ian polynomial $\qbin{N+4}{ 4}_q$ is the same as the difference between partition of a number half the size into at most three parts}. 
Line~\eqref{repp(2N-1-2,4,N)eq2} of Proposition~\ref{4isD3} can be read as {\em four of the five coefficients {\em in the middle} of any Gau\ss ian polynomial $\qbin{N+4}{ 4}_q$ are always the same}. 
Another interpretation of~\eqref{repp(2N-1-2,4,N)eq2} comes from an independent proof by D.~Burde and F.~Wagemann: {\em The adjoint $\textswab{s:l}_2(\C)$-module $V_2$ does not occur in $\Lambda^4(V_{k+3})$ for all $k\geq 1$}~\cite{BurdeWagemann}.  
In Section~\ref{subsection4} we show that Proposition~\ref{4isD3} is a quick corollary to a very general result.  


Regardless of interpretations, the reader can examine Table~\ref{p(n,4,N)TABLE} for some reassuring evidence supporting Proposition~\ref{4isD3}.

\section{Main results}
\label{establishing}

Here we present our formulas for the perpendicular generating functions
$$\sum_{N=0}^{\infty} p\left(\left\lfloor \frac{mN}{2}
  \right\rfloor-A,m,N\right)z^{N}.$$
For the statement of the results, we need to introduce the
{\it mod-$s$-Huffing operator}~$H_s$, which acts on
polynomials $P(z)=\sum_{i=0}^d a_iz^i$ by
$$(H_sP)(z):=\sum_{i=0}^{\fl{d/s}}a_{is}z^i.
$$
In other words, the mod-$s$-Huffing operator~$H_s$
takes a polynomial $P(z)$ and builds 
a new polynomial $(H_sP)(z)$ by taking every $s$-th coefficient
of~$P(z)$ and ignoring all the other coefficients. It is easy to see
(and well-known) how to express the resulting polynomial~$(H_sP)(z)$
in terms of the original polynomial~$P(z)$.

\begin{lemma} \label{lem:s}
For a polynomial $P(z)$, we have
$$
(H_sP)(z)=\frac {1} {s}\sum_{\ell=0}^{s-1}P(\om_s^\ell z^{1/s}),
$$
where $\om_s$ is a primitive $s$-th root of unity.
\end{lemma}

\medskip
If $m$ is even, we have the following result.

\begin{theorem} \label{thm:even}
Let $M$ be a positive integer and $a$ and $r$ be nonnegative integers.
Furthermore define $A_M:=\lcm(1,2,\dots,M)$. 
Then the partition generating function
$\sum_{N=0}^{\infty} p\left(MN
  -(A_M a+r),2M,N\right)z^{N}$ is equal to
$$
\frac {\text{\em Num}_e(M,r)} {(1-z^2)(z;z)_{2M-1}},
$$
where the numerator $\text{\em Num}_e(M,r)$ is given by
\begin{equation} \label{eq:even} 
\sum_{j=1}^M(-1)^{M - j}
z^{A_M a/j} H_j\left(
z^{r + \binom {M - j + 1}2}
  \frac {(1 - z^{2 j})\,(z^j;z^j)_{2M-1}} {(z;z)_{2M}}
           \Qbin{2 M}{ M - j}_z\right).
\end{equation}
\end{theorem}

\begin{remark}
(1) The proof of this theorem is given in Section~\ref{CKproof}.
In particular, it follows from that proof that the expression
in~\eqref{eq:even} to which the Huffing operator~$H_j$ is applied is
indeed a polynomial in~$z$.

\medskip
(2) As the theorem shows, the generating function
$\sum_{N=0}^{\infty} p\left(MN
-A,2M,N\right)z^{N}$ is rational, and all the roots
of the denominator are roots of unity. It is a well-known fact
(cf.\ Proposition~\ref{prop:2}) that these properties imply that
the coefficients of the considered power series are quasipolynomial
(see Definition~\ref{def:quasi}). Consequently, the partition
numbers $p\left(MN -A,2M,N\right)$ are quasipolynomial in~$N$.
\end{remark}

\medskip
For the case where $m$ is odd, we have the following result.

\begin{theorem} \label{thm:odd}
Let $M$ be a positive integer and $a$ and $r$ be nonnegative integers.
Furthermore define $B_M:=\lcm(1,3,\dots,2M-1)$. 
Then the partition generating function
$\sum_{N=0}^{\infty} p\left(\left\lfloor\frac {(2M-1)N} {2}\right\rfloor
  -(B_M a+r),2M-1,N\right)z^{N}$ is equal to
$$
\frac {\text{\em Num}_o(M,r)} {(1-z)(z^2;z^2)_{2M-2}},
$$
where the numerator $\text{\em Num}_o(M,r)$ is given by
\begin{multline} \label{eq:odd} 
\sum_{j=1}^M(-1)^{M - j}
z^{2B_M a/(2j-1)}\\
\times H_{2j-1}\left(
z^{2r + 2\binom {M - j + 1}2}
\frac {(1 - z^{2 j - 1}) \,
  (z^{2(2j-1)};z^{2(2j-1)})_{2M-2}}
   {(1-z)\,(z^4;z^2)_{2M-2}}
           \Qbin{2 M-1}{ M - j}_z\right).
\end{multline}
\end{theorem}

\begin{remark}
(1) The proof of this theorem is also given in Section~\ref{CKproof}.
Again, it follows from that proof that the expression
in~\eqref{eq:odd} to which the Huffing operator~$H_{2j-1}$ is applied is
indeed a polynomial in~$z$.

\medskip

\medskip
(2) Similarly as before, the theorem shows
that the perpendicular generating function
$\sum_{N=0}^{\infty} p\left(\left\lfloor\frac {(2M-1)N} {2}\right\rfloor
  -A,2M-1,N\right)z^{N}$ is rational, and all the roots
of the denominator are roots of unity. As above, the consequence is that
the partitions
numbers $p\left(\left\lfloor\frac {(2M-1)N} {2}\right\rfloor
  -A,2M-1,N\right)$ are quasipolynomial in~$N$.

The quasipolynomial for $p(n,3,N)$ was first computed in~\cite{LSAMPpn3N}.  
It is better described as six quasipolynomials of period 6.  
These 36 formulas can be found in Appendix~A of~\cite{LSAMPpn3N}.  
\end{remark}

We have implemented the formulas in Theorems~\ref{thm:even}
and~\ref{thm:odd} in {\sl Mathematica} which allowed us to compute
the generating functions
$\sum_{N=0}^{\infty} p\left(\left\lfloor\frac {mN} {2}\right\rfloor
-A,m,N\right)z^{N}$ for\break $m=1,2,\dots,12$. It is also possible
to compute these generating functions for $A$ in some specific
congruence class modulo~$A_M$ respectively~$B_M$ for values of~$m$ far
beyond~20. (Clearly, since $A_M$ and $B_M$ grow quickly, the {\it number}
of congruences classes becomes enormous for large~$m$.)
The implementation is available in the notebook file\break
{\tt orthview.nb} accompanying this article.

\medskip
For illustration, we present here the results
implied by Theorems~\ref{thm:even} and~\ref{thm:odd} for $m=1,2,3,4,5,6$.
Keeping Remark \ref{A} in mind, we need only consider $A\geq 0$ and the results follow for $A<0$.  

\begin{proposition}\label{p(n,1,N)}
For all $A\ge0$, 
\begin{equation}\label{p(n,1,N)eq}
\displaystyle\sum_{N=0}^{\infty}p\big(\left\lfloor N/2\right\rfloor-A,1,N\big)z^N = \frac{z^{2A}}{1-z}.
    \end{equation}
\end{proposition}

\begin{proposition}\label{p(n,2,N)}
For all $A\ge0$, 
\begin{equation}\label{p(n,2,N)eq}
\displaystyle\sum_{N=0}^{\infty}p\big(N-A,2,N\big)z^N = \frac{z^{A}}{\left(1-z\right)\left(1-z^2\right)}
    \end{equation}
\end{proposition}

\begin{proposition}\label{mainprop}
For all $A\ge 0$,  
\begin{equation}
    \sum_{N=0}^{\infty}p\left(\left\lfloor\frac{3N}{2}\right\rfloor\!-\!A,3,N\! \right)\!z^N\! =\begin{cases}
\displaystyle
    \frac{z^{2a}\left(1+z^2+z^3-z^{4a+2}\right)}{(1-z)(1-z^2)(1-z^4)},&
    \text{if }A=3a,\\
\displaystyle
    \frac{z^{2a+1}\left(1+z+z^3-z^{4a+3}\right)}{(1-z)(1-z^2)(1-z^4)},&
    \text{if }A=3a+1,\\
\displaystyle
    \frac{z^{2a+2}\left(1+z+z^2-z^{4a+4}\right)}{(1-z)(1-z^2)(1-z^4)},&
    \text{if }A=3a+2.
    \end{cases}
\label{30}
    \end{equation}
\end{proposition}

\begin{proposition}\label{C_4b}
For all $A\ge 0$,
\begin{equation}
  \sum_{N=0}^{\infty}p\big(2N-A,4,N\big)z^N =\begin{cases} 
\displaystyle
  \frac{z^{a}\left(1+z^2-z^{a+1}\right)}{\left(1-z\right)^2\left(1-z^2\right)\left(1-z^3\right)},&
  \text{if }A=2a, \\ 
\displaystyle
  \frac{z^{a+1}\left(1+z-z^{a+1}\right)}{\left(1-z\right)^2\left(1-z^2\right)\left(1-z^3\right)},&
\text{if }a=2a+1.  
  \end{cases}
  \label{C_4eq}
    \end{equation}
\end{proposition}


The rational functions corresponding to the perpendicular generating functions for\break $p\left\{\left(\left\lfloor \frac{5N}{2} \right\rfloor-A,5,N\right)\right\}_{N,A\ge0}$ and $\{p\left(3N-A,6,N\right)\}_{N,A\ge0}$ can be found in the appendix of this article.  
We note that $\sum_{N=0}^{\infty} p\left(\left\lfloor \frac{5N}{2} \right\rfloor-A,5,N\right)z^{N}$ is described by 15 rational functions, while  $\sum_{N=0}^{\infty} p\left(3N-A,6,N\right)z^{N}$ is described by six.

\section{Unimodality of $\protect\qbin{N+m}{ m}_q$ for $m=0,1,2,3,4,5,6$.}\label{uni}

There are several proofs of the unimodality of Gau\ss ian polynomials.  
J.~J.~Sylvester \cite{Sylvester2} was the first to prove it in~1878. 
 I.~J.~Schur's proof~\cite{SchurUni} employs the theory of invariants.  
Proctor~\cite{ProctorUni} offered a proof with a telling title: {\em Solution of two difficult combinatorial problems with linear algebra}.  
O'Hara's proof~\cite{OHara} is the first proof based on a combinatorial understanding of the Gau\ss ian polynomial.  
So celebrated is this proof that not only Bressoud~\cite{BressoudUni}, but also Zeilberger \cite{ZeilbergerOhara} wrote follow-up papers offering stream-lined versions of O'Hara's proof.  
In fact, Zeilberger wrote  other follow-up papers; \cite{StantonZeilbergerUni,ZeilbergerUni2}.  
Recent work, \cite{PakPanova} and~\cite{Uncu}, on strict unimodality of Gau\ss ian polynomials is also very interesting.  

In this section we use the generating function formulas from
Propositions~\ref{p(n,2,N)}--\ref{C_4b} and the appendix
to provide new proofs
of the unimodality of the $q$-binomial coefficients 
$\qbin{N+2}2$, $\qbin{N+3}3$, $\qbin{N+4}4$, $\qbin{N+5}5$,
and $\qbin{N+6}6$. For the sake of completeness, we also briefly discuss
$\qbin{N+0}0$ and $\qbin{N+1}1$.

We begin by introducing notation for differences of partition
functions. This notation will also be used in Section~\ref{ID}.

\begin{definition}[{\sc Partition difference functions}]\label{DeltapartitionsA}
For any $x\in\Z$ we define the following functions:
\begin{itemize}
    \item     $\Delta_x p(n,m) = p(n,m) - p(n-x,m)$
    \item $\Delta_x p(n,m,N) = p(n,m,N) - p(n-x,m,N)$ 
\end{itemize}

Whenever $x=1$, we omit the subscript.  
Additionally, for any $n$, if $x=0$, then the value of the difference functions is zero.  
\end{definition}

\subsection{The coefficients of $\protect\qbin{N}{ 0}_q$, $\protect\qbin{N+1}{ 1}_q$ and $\protect\qbin{N+2}{ 2}_q$ are unimodal}

Since $\qbin{N}{ 0}_q=1$ for all $N$, unimodality follows trivially.  

We note that the coefficients of the Gau\ss ian polynomials
$\qbin{N+1}{ 1}_q$ are all 1, and therefore unimodality is settled in
this case as well.   

\begin{proposition}\label{U2}
 The coefficients of\/ $\qbin{N+2}{ 2}_q$ are unimodal.
\end{proposition}
\begin{proof}
From \eqref{p(n,2,N)eq}, we obtain
\begin{align}\notag
\sum_{N=0}^{\infty}\Delta p\big(N-A,2,N\big)z^N
&=\sum_{N=0}^{\infty}\big(p\big(N-A,2,N\big)-p\big(N-A-1,2,N\big)\big)z^N\\
&= \frac{z^A}{1-z^2} = 
    \sum_{N=0}^{\infty}z^{2N+A}.
\label{p2}
\end{align}
For any $A\ge0$, the series on the right-hand side of~\eqref{p2} has nonnegative coefficients.  
Thus, by symmetry of Gau\ss ian polynomials, we have shown that the coefficients of $\qbin{N+2}{ 2}_q$ are unimodal.
\end{proof}

\subsection{The coefficients of $\protect\qbin{N+3}{ 3}_q$ are unimodal}
  
\begin{proposition}\label{U3}
 The coefficients of\/ $\qbin{N+3}{ 3}_q$ are unimodal.
\end{proposition}

\begin{proof} 
We consider the differences of the generating functions for $p\big(\left\lfloor3N/2\right\rfloor-A,3,N \big)$ in Proposition~\ref{mainprop} to show that $p\big(\left\lfloor3N/2\right\rfloor-A,3,N \big)\ge p\big(\left\lfloor3N/2\right\rfloor-(A+1),3,N \big)$ for all $A$ and $N$.  
This will be done by computing the difference of successive generating functions for $p\big(\left\lfloor3N/2\right\rfloor-A,3,N \big)$ and then showing that the coefficients of the resulting generating function are nonnegative.  

For brevity we will compute the difference of the generating functions in
the first two cases on the right-hand side of~\eqref{30}.  
The remaining computations and verifications are done identically and
so are omitted.  

We have
\begin{align}
\notag
\sum_{N=0}^{\infty}\Delta p\big(\left\lfloor3N/2\right\rfloor-3a,3,N \big)z^N 
&= \frac{z^{2a}(1-z+z^2-z^{4a+2})}{(1-z)(1-z^4)}\\
&= \frac{z^{2a}}{1-z^4}
+ \frac{z^{2a+2}\sum_{i=0}^{a-1}z^{4i}}{1-z}.
\label{U30}
\end{align}
After expansion of geometric series on the right-hand side, it is
obvious that all coefficients in this power series are nonnegative.
Thus, the coefficients of $\qbin{N+3}{ 3}_q$ are unimodal.  
\end{proof}     

\subsection{The coefficients of $\protect\qbin{N+4}{ 4}_q$ are unimodal}\label{Uni3}

\begin{proposition}\label{U4}
  The coefficients of\/ $\qbin{N+4 }{ 4}_q$ are unimodal.
\end{proposition}

\begin{proof}
Working from the first two cases in~\eqref{C_4eq} in
Proposition~\ref{C_4b}, we obtain
\begin{equation}\label{U41}
      \sum_{N=0}^{\infty}\Delta p\big(2N-2a,4,N\big)z^N 
      = \frac{z^{a}(1-z^{a+1})}{(z;z)_3} 
      = \frac{z^{a}\sum_{i=0}^a z^i}{(z^2;z)_2} 
  \end{equation}
  and 
  \begin{equation}\label{U42}
      \sum_{N=0}^{\infty}\Delta p\big(2N-(2a+1),4,N\big)z^N  
       = \frac{z^{a+2}(1-z^{a})}{(z;z)_3} 
       = \frac{z^{a+2}\sum_{i=0}^{a-1}z^i}{(z^2;z)_2} .
  \end{equation}
After expansion of geometric series on the right-hand sides, it is
obvious that all coefficients in these power series are nonnegative.
Thus, for all $N\ge0$, the coefficients of $\qbin{N+4 }{ 4}_q$ are unimodal.
\end{proof}

\subsection{The unimodality of $\protect\qbin{N+5}{ 5}_q$ and $\protect\qbin{N+6}{ 6}_q$}  

Our proofs of unimodality of $\qbin{N+5}{ 5}_q$ and $\qbin{N+6}{ 6}_q$ follow the same strategy as with $\qbin{N+2}{ 2}_q$, $\qbin{N+3}{ 3}_q$ and $\qbin{N+4}{ 4}_q$.
As before, we consider the difference of rational functions to obtain a generating function and corresponding rational functions for both $\sum_{N=0}^{\infty}\Delta p\big(\left\lfloor5N/2\right\rfloor-A,5,N \big)z^N$ and $\sum_{N=0}^{\infty}\Delta p\big(3N-A,6,N \big)z^N$.  
Since this involves 15~cases for the first generating function and
6~cases for the second, we content ourselves with discussing
just one cases. All other cases are treated analogously.

We choose $\sum_{N=0}^{\infty}\Delta
p\big(\left\lfloor5N/2\right\rfloor-A,5,N \big)z^N$ with $A=15a$
as our example.
Working from items \eqref{A5.0} and~\eqref{A5.1}
in Proposition~\ref{GF5}
in the appendix, we obtain
\begin{multline}\label{15a-(15a+1)}
  \sum_{N=0}^{\infty}\Delta p\left(\left\lfloor \frac{5N}{2} \right\rfloor-15a,5,N\right)z^N=\\
    z^{6a}\left(-z^{4a+12}-z^{4 a+1}-z^{4 a+2}-z^{4 a+3}-3 z^{4 a+4}-2 z^{4 a+5}-5 z^{4 a+6}-z^{4 a+7}-5 z^{4 a+8}\right.\\
\left.-2 z^{4 a+9}-3 z^{4 a+10}-z^{4 a+11}-z^{4 a+13}+z^{24 a+6}+z^{24 a+10}+2 z^{12}+z^{11}+2 z^{10}+2 z^9\right.\\
\left.+4 z^8+3 z^7+2 z^6+3 z^5+3 z^4+z^3+z^2+1\right)\times\frac{1}{(1-z)(1-z^4)(1-z^6)(1-z^8)}.
\end{multline}
We regroup the numerator polynomial in the form
\begin{multline*}
  -z^{4a+12}-z^{4 a+1}-z^{4 a+2}-z^{4 a+3}-3 z^{4 a+4}-2 z^{4 a+5}-5 z^{4 a+6}-z^{4 a+7}-5 z^{4 a+8}\\
-2 z^{4 a+9}-3 z^{4 a+10}-z^{4 a+11}-z^{4 a+13}+z^{24 a+6}+z^{24
  a+10}+2 z^{12}+z^{11}+2 z^{10}+2 z^9\\
+4 z^8+3 z^7+2 z^6+3 z^5+3 z^4+z^3+z^2+1\\
=(1+z^4)(1 - z^{4 a + 4}) (1 - z^{20 a + 2}) + 
z^3 (1+z^4)(1 - z^{4 a}) (1 - z^{16 a - 1}) \kern3cm\\
+ z^5 (1+z^4)(1 - z^{4 a}) (1 - z^{12 a - 3}) + 
z^4 (1+z^4)(1 - z^{4 a}) (1 - z^{8 a - 2}) \\
+  z^5 (1+z^4)(1 - z^{4 a}) (1 - z^{4 a - 3})
\\
+(z^2-z^{4 a+1})
+(z^4-z^{4 a+4})
+2( z^6- z^{4 a+6})
+2 (z^7- z^{4 a+6})
+3 (z^8- z^{4 a+8})\\
+2 (z^{10}- z^{4 a+10})
+(z^5- z^{4 a+10})
+(z^{11}-z^{4 a+11})
+(z^{12}-z^{4a+12}))
+(z^{12}-z^{4 a+13}).
\end{multline*}
Then, as long as $a\ge1$, the first five summands on the right-hand
side are divisible by $(1-z)(1-z^4)$, while the remaining summands
are divisible by $1-z$. After division, in each case a polynomial
with nonnegative coefficients remains. For example, with denominator
in~\eqref{15a-(15a+1)} included, for the first summand we have
$$
\frac {(1+z^4)(1 - z^{4 a + 4}) (1 - z^{20 a + 2}) }
      {(1-z)(1-z^4)(1-z^6)(1-z^8)}
=
\frac {(1+z^4)
  \left(\sum_{i=0\vphantom{j}}^{a} z^{4i}\right)
  \left(\sum_{j=0}^{20a+1}z^j\right)}
      {(1-z^6)(1-z^8)}.
$$
This shows that the power
series in~\eqref{15a-(15a+1)} is a series with nonnegative
coefficients.

For $a=0$, the numerator in~\eqref{15a-(15a+1)} reduces to
$$
1 - z + z^5 - 2 z^6 + 2 z^7 - z^8 + z^{12} - z^{13}
=(1 - z) (1 - z^6) + (1 - z) \left(z^5 + z^7 + z^{12}\right).
$$
Again, the last regrouping of terms
shows nonnegativity of the coefficients of the power series
in~\eqref{15a-(15a+1)}.

\medskip
With the complete collection of generating functions for $\Delta
p\left(\left\lfloor \frac{5N}{2} \right\rfloor-A,5,N\right)$
established, we may proceed similarly in the other 14~cases.

\medskip
This same process is repeated to prove the unimodality of $\qbin{N+6}{ 6}_q$.

\section{Difference Partition identities related to unimodality}\label{ID}

This section is inspired by some of the identities in the previous
section. For example, we may start with~\eqref{U41} and~\eqref{U42}
and observe that
\begin{equation}
      \sum_{N=0}^{\infty}\Delta p\big(2N-2a,4,N\big)z^N 
      = \frac{z^{a}(1-z^{a+1})}{(z;z)_3} 
      \\ = z^{a}\sum_{n=0}^{\infty}\Delta_{a+1} p\big(n,3\big)z^n
  \end{equation}
  and 
  \begin{equation}
      \sum_{N=0}^{\infty}\Delta p\big(2N-(2a+1),4,N\big)z^N  
       = \frac{z^{a+2}(1-z^{a})}{(z;z)_3} \\ 
       = z^{a+2}\sum_{n=0}^{\infty}\Delta_{a} p\big(n,3\big)z^n.
\end{equation}
This yields a surprising connection between the partition numbers
$p\big(2N-A,4,N\big)$ and $p(N,3)$, namely
\begin{align*}
\Delta p\big(2N-2a,4,N\big)&=\Delta_{a+1} p(N,3),\\
\Delta p\big(2N-(2a+1),4,N\big)&=\Delta_a p(N,3).
\end{align*}

In \cite{LSAMPpn3N} a handful of first differences identities of coefficients either at or near the center of $\qbin{N+3}{ 3}_q$ and $\qbin{N+4}{ 4}_q$ were established.  
With our collection of perpendicular generating functions we expand and in some cases generalize those results for $\qbin{N+3}{ 3}_q$ and $\qbin{N+4}{ 4}_q$.  
We then go on to prove other surprising difference identities for Gau\ss ian polynomial coefficients either at or near the center of $\qbin{N+5}{ 5}_q$ and  $\qbin{N+6}{ 6}_q$.   

For completeness we include results for $m=1,2$.  


\subsection{$\protect\qbin{N+1}{ 1}_q$ difference identities}

\begin{proposition}  For nonnegative integers $N$ and $A$, we have
    \begin{equation}
        \Delta p\left(\left\lfloor\frac{N}{2}\right\rfloor-A,1,N\right)=\Delta p(N,1)=0.
    \end{equation}
\end{proposition}    
\begin{proof}
    Since the coefficients of $\qbin{N+1}{ 1}_q$ and $p(N-A,1)$ are 1, the result follows.
\end{proof}

\subsection{$\protect\qbin{N+2}{ 2}_q$ difference identities}

\begin{proposition}For nonnegative integers $N$ and $A$, we have
    \begin{equation}
        \Delta p(N-A,2,N)=\Delta p(N-A,2)=\begin{cases}
1, \mbox{ if } N-A \mbox{ is even,} \\
0, \mbox{ if } N-A \mbox{ is odd.} 
\end{cases} 
    \end{equation}
\end{proposition}

\begin{proof}  
From Proposition~\ref{p(n,2,N)} we get
\begin{equation}\label{p(n,2,N)eqb}
\displaystyle\sum_{N=0}^{\infty}\Delta p\big(N-A,2,N\big)z^N = \frac{z^{A}}{\left(1-z^2\right)} = \sum_{N=0}^{\infty}\Delta p\big(N-A,2\big)z^N.
\end{equation}
    The second equality comes from the fact that $p(n,2)=\lfloor\frac{n+2}{2}\rfloor$.
\end{proof}

\subsection{$\protect\qbin{N+3}{ 3}_q$ difference identities}

Here we take the opportunity to display three complete collections of results.  
Half of the individual lines within \eqref{alfalfa3-0} and~\eqref{alfalfa3-1} below appeared in~\cite{LSAMPpn3N}, \cite{HahnPAMS2016}, and~\cite{Hahn2017pn3N} when taken together.  
None of those publications had the complete results of either~\eqref{alfalfa3-0} or~\eqref{alfalfa3-1}.  
The contents of~\eqref{alfalfa3-2} are entirely new. 
\begin{proposition}\label{alfalfa3}
For nonnegative integers $N$, we have
\begin{equation}\label{alfalfa3-0}
\Delta p\left(\left\lfloor\frac{3N}{2}\right\rfloor,3,N \right)=
\begin{dcases}
1,& \mbox{ for } N\equiv0\pmod4,  \\
0,& \mbox{ for } N\equiv1\pmod4, \\
0,& \mbox{ for } N\equiv2\pmod4, \\
0,& \mbox{ for } N\equiv3\pmod4.
\end{dcases}    
\end{equation}
\begin{equation}\label{alfalfa3-1}
\Delta p\left(\left\lfloor\frac{3N}{2}\right\rfloor-1,3,N \right)=
\begin{dcases}
0,& \mbox{ for } N\equiv0\pmod4,  \\
1,& \mbox{ for } N\equiv1\pmod4, \\
1,& \mbox{ for } N\equiv2\pmod4, \\
1,& \mbox{ for } N\equiv3\pmod4.
\end{dcases}    
\end{equation}
\begin{equation}\label{alfalfa3-2}
\Delta p\left(\left\lfloor\frac{3N}{2}\right\rfloor-2,3,N \right)=
\begin{dcases}
1,& \mbox{ for } N\equiv0\pmod4, \\
1,& \mbox{ for } N\equiv1\pmod4, \\
0,& \mbox{ for } N\equiv2\pmod4, \\
1,& \mbox{ for } N\equiv3\pmod4.
\end{dcases}    
\end{equation}
\end{proposition}

\begin{proof}
We prove \eqref{alfalfa3-0}.  
The other results are proved similarly and the corresponding proofs
are omitted for brevity.  

Set $a=0$ in \eqref{U30}. The result is
$$\sum_{N=0}^{\infty}\Delta p\left(\left\lfloor\frac{3N}{2}\right\rfloor,3,N \right)z^n = \frac{1}{1-z^4} = \sum_{N=0}^{\infty}z^{4N}.$$
Thus \eqref{alfalfa3-0} is proved.  
\end{proof}

\subsection{$\protect\qbin{N+4}{ 4}_q$ difference identities}\label{subsection4}

In \cite{LSAMPpn3N}, a quasipolynomial was established for the central coefficient $p(2N,4,N)$ for all $N$.  
With our techniques, in this article we are able to go beyond that result and describe any coefficient or any relation among the coefficients of $\qbin{N+4}{ 4}_q$.  

The following result for first differences of $p(n,4,N)$ follows immediately from~\eqref{U41} and~\eqref{U42}.  
It may be of interest to examine second differences, $\Delta_2 p(n,4,N)$, and beyond.  
\begin{proposition}\label{4cor}
Let $N,a\ge0$.  
Then
\begin{equation}\label{4coreqE}
        \Delta p\big(2N-2a,4,N\big) 
      = \Delta_{a+1} p\big(N-a,3\big)
    \end{equation}
    and
\begin{equation}\label{4coreqO}
   \Delta p\big(2N-(2a+1),4,N\big)=\Delta_{a} p\big(N-2-a,3\big).
   \end{equation}
\end{proposition}
\begin{proof}
    The result follows immediately from \eqref{U41} and \eqref{U42}.
\end{proof}
Proposition~\ref{4isD3}, introduced in Section~\ref{sectionstructure}, is restated as part of a corollary to Proposition~\ref{4cor}.  
\begin{corollary}\label{4isD3c}
Let $N\ge 0$. 
Then
\begin{equation}
\Delta p\big(2N,4,N \big)= \Delta p\big(N,3\big)
\end{equation}   
and
\begin{equation}
        \Delta p(2N-1,4,N) = 0.
    \end{equation}
\end{corollary}
\begin{proof}
    Set $a=0$ in lines \eqref{4coreqE} and \eqref{4coreqO} of Proposition~\ref{4cor}. 
\end{proof}

We also have a very general Ramanujan-style partition congruence result that extends to all primes.  
\begin{proposition}\label{pn3congruences}
  Let $\ell$ be a prime.  
  Whenever $a=6\ell j-1$, then
  \begin{equation}\label{ex3}
      \Delta p\big(2N-2a,4,N\big) 
      \equiv0\pmod{\ell}.
  \end{equation}
Whenever $a=6\ell j$, then
  \begin{equation}
      \Delta p\big(2N-(2a+1),4,N\big) 
      \equiv0\pmod{\ell}.
  \end{equation}
\end{proposition}

\begin{proof}
    The proof of each line in Proposition~\ref{pn3congruences} follows
    from the fact that for any $j\ge0$, we have 
    $$\Delta_{6\ell j}p(n,3) 
    \equiv0\pmod{\ell},$$ 
    which follows from Theorem~1 proved in~\cite{KronholmINTEGERS}.  
\end{proof}

We display an example of line~\eqref{ex3} from Proposition~\ref{pn3congruences}.   
\begin{example}
Let $\ell=5$ and $j=1$ so that in line~\eqref{ex3} we have $a=5\cdot6=30$.  
    Now with $N=67$ we compute:
    $$\Delta p(2(67)-2(29),4,67)=p(76,4,67)-p(75,4,67)=3648-3518=130\equiv0\pmod5.$$
Equivalently, 
$$\Delta p(2(67)-2(29),4,67)=\Delta_{30}p(38,3)=p(38,3)-p(8,3)=140-10=130\equiv0~\,\,(\text{mod }5).$$
\end{example}

\subsection{$\protect\qbin{N+5}{ 5}_q$ difference identities}

\begin{proposition}\label{5is3Prop}
For all nonnegative integers $N$, we have
\begin{equation}\label{5is3eq}
 \Delta p\left(\left\lfloor \frac{5N}{2} \right\rfloor,5,N\right)=   \begin{dcases}
p(n,3), &\mbox{ for } N=4n, \\
p(n-1,3)+p(n-3,3), &\mbox{ for } N=4n+1, \\
p(n-4,3),& \mbox{ for } N=4n+2, \\
p(n-1,3)+p(n-2,3), &\mbox{ for } N=4n+3. 
\end{dcases} 
\end{equation}   
\end{proposition}

\begin{proof}
Set $a=0$ in \eqref{15a-(15a+1)} to obtain the following rational function:
\begin{align*}
   \sum_{N=0}^{\infty}\Delta p\left(\left\lfloor \frac{5N}{2}
   \right\rfloor,5,N\right)z^N
&= \frac{1+z^5-z^6+z^7+z^{12}}{\left(1-z^4\right) \left(1-z^6\right) \left(1-z^8\right)}\\
&= \frac{1+z^5+z^7+z^{11}+z^{13}+z^{18}}
   {\left(1-z^4\right) \left(1-z^8\right) \left(1-z^{12}\right)}.
\end{align*}
From the last expression, it is easy to extract the coefficients
of~$z^N$ with $N$ in a particular residue class modulo~4. To be
precise, the generating function for the differences
$p\left(\left\lfloor \frac{5N}{2} \right\rfloor,5,N\right)$
with $N\equiv0$~(mod~4) equals
$$
\frac{1}
   {\left(1-z^4\right) \left(1-z^8\right) \left(1-z^{12}\right)},
$$
the generating function for those with $N\equiv1$~(mod~4) equals
$$
\frac{z^5+z^{13}}
   {\left(1-z^4\right) \left(1-z^8\right) \left(1-z^{12}\right)},
$$
the generating function for those with $N\equiv2$~(mod~4) equals
$$
\frac{z^{18}}
   {\left(1-z^4\right) \left(1-z^8\right) \left(1-z^{12}\right)},
$$
and the generating function for those with $N\equiv3$~(mod~4) equals
$$
\frac{z^7+z^{11}}
   {\left(1-z^4\right) \left(1-z^8\right) \left(1-z^{12}\right)}.
$$
In view of
\begin{equation}\label{p(n,3)gf}
    \frac{1}{(1-z)(1-z^2)(1-z^3)}
    =\sum_{n=0}^{\infty}p\big(n,3\big)z^n,
\end{equation}
the claims in \eqref{5is3eq} are now obvious.
 \end{proof}

\subsection{$\protect\qbin{N+6}{ 6}_q$ difference identities}

An analysis of $\Delta p(n,6,N)$ results in further compelling identities.  
We offer just one here.
The proof is similar to that of Proposition~\ref{5is3Prop} and is omitted.
\begin{proposition}\label{N+61stdiff}
  For all positive integers $N$, we have
    \begin{equation}
 \Delta p\left(3N,6,N\right)=   \begin{dcases}
p(N\mid\mbox{parts from the set }\{1,2,3,5\}),  & \mbox{ for $N$ even,}\\
p(N-7\mid\mbox{parts from the set }\{1,2,3,5\}), &\mbox{ for $N$ odd.}
\end{dcases} 
\end{equation} 
\end{proposition}
We note that $p(n\mid\mbox{parts from the set }\{1,2,3,5\})$ is equivalent to $\Delta_4 p(n,5)$.

\section{Proofs of Theorems \ref{thm:even} and \ref{thm:odd}}\label{CKproof}

In this section, we present the proofs of our main results,
Theorems~\ref{thm:even} and~\ref{thm:odd}.

To begin with, we provide the formal definition of a quasipolynomial
sequence and then quote the standard theorem that characterizes
the generating functions for such sequences.
\begin{definition} \label{def:quasi}
A sequence $\{f(n)\}_{n\ge0}$ is {\em quasipolynomial} if there exist $d$ polynomials $f_0{(n)},\ldots, f_{d-1}{(n)}$ such that 
\[
 f(n)  =
\begin{cases}
f_{0}(n), & \text{if $n\equiv0$ (mod }d),\\
f_{1}(n), & \text{if $n\equiv1$ (mod }d),\\
~~~\vdots & ~~~~~~\vdots\\
f_{d-1}(n), & \text{if $n\equiv{d-1}$ (mod }d),
\end{cases}
\]
for all $n\in\Z$. The polynomials $f_i$ are called the constituents of~$f$ and the number of them, $d$, is the period of~$f$. 
\end{definition}

\begin{example}\label{QPp(n,3)}
For example, the infinite sequence $\{p(n,3)\}_{n\ge0}$ is described by a quasipolynomial of period six.  
Namely, for nonnegative integers $n$, we have
\begin{equation}\label{Ep(6k+0,3)}
p(n,3)=\begin{cases}
1\tbinom{k+2}{2}+4\tbinom{k+1}{2}+1\tbinom{k}{2}=3k^2+3
k +1,&\text{if }n=6k,\\
1\tbinom{k+2}{2}+5\tbinom{k+1}{2} =3k^2+4k+1,&\text{if }n=6k+1,\\
2\tbinom{k+2}{2}+4\tbinom{k+1}{2} =3k^2+5k+2,&\text{if }n=6k+2,\\
3\tbinom{k+2}{2}+3\tbinom{k+1}{2} =3k^2+6 k+ 3,&\text{if }n=6k+3,\\
4\tbinom{k+2}{2}+2\tbinom{k+1}{2} =3k^2+ 7k+4,&\text{if }n=6k+4,\\
5\tbinom{k+2}{2}+1\tbinom{k+1}{2} =3k^2+8k+5,&\text{if }n=6k+5.
\end{cases} 
\end{equation}    
\end{example}
\begin{remark}
For further information on how quasipolynomials of this variety
are computed and the associated geometry associated with integer
partitions, 
see~\cite{GuptaTechnique,BEK-6j-1,LSAMPpn3N,gupta}.   
\end{remark}

Next we recall \cite[Prop.~4.4.1]{StanBT}.

\begin{proposition} \label{prop:2}
A sequence $\{f(n)\}_{n\ge0}$ is quasipolynomial if and only if
its generating function $\sum_{n\ge0}f(n)z^n$ is rational in~$z$
and all roots of the denominator of the rational function are roots of unity.
\end{proposition}

Now, by definition, we have
$$
p\left(\left\lfloor\frac {mN} {2}\right\rfloor-A,m,N\right)
=\coef{q^{\fl{mN/2}-A}}\bmatrix m+N\\m\endbmatrix_q,
$$
where $\coef{q^{e}}h(z)$
denotes the coefficient of $q^{e}$ in the polynomial or formal
power series $h(q)$.

\medskip
First, we consider the case where $m$ is even, say $m=2M$.
We are then talking of
$$
p\left({MN} -A,2M,N\right)
=\coef{q^{{MN}-A}}\bmatrix 2M+N\\2M\endbmatrix_q.
$$
We claim that, for fixed $M$ and $A$, this is quasipolynomial in~$N$.
To see this, we write
\begin{equation} \label{eq:1} 
\bmatrix 2M+N\\2M\endbmatrix_q
=\frac {
\prod _{j=1} ^{2M}(1-q^{N+j})} {
\prod _{j=1} ^{2M}(1-q^j)}
=\frac {\sum_{i=0}^{2M}c_i(q)q^{iN}}
{\prod _{j=1} ^{2M}(1-q^j)},
\end{equation}
for certain polynomials $c_i(q)$, $i=0,1,\dots,2M$.
Since
$$
\coef{q^N}\frac {1}{\prod _{j=1} ^{2M}(1-q^j)}
$$
is a quasipolynomial in $N$, the same is true for
$$
\coef{q^{sN-B}}\frac {1}{\prod _{j=1} ^{2M}(1-q^j)}
$$
for any fixed $s$ and $B$. In view of~\eqref{eq:1}, this confirms our
claim.

Consequently, by Proposition~\ref{prop:2},
we know {\it a priori} that the generating function
$$
\sum_{N\ge0}p\left({MN} -A,2M,N\right)z^N
$$
is a rational function in $z$
and all roots of the denominator of the rational function are
roots of unity.

\medskip
Next we express this generating function in terms of a complex
contour integral. Namely, we have
\begin{align*}
\sum_{N\ge0}p\left({MN} -A,2M,N\right)z^N
&=
\sum_{N\ge0}z^N
\coef{q^{{MN}-A}}\bmatrix 2M+N\\2M\endbmatrix_q\\
&=
\sum_{N\ge0}z^N
\frac {1} {2\pi i}\int_{\mathcal C}\frac {dq} {q^{MN-A+1}}
\frac {(q^{2M+1};q)_{N}} {(q;q)_{N}},
\end{align*}
where $\mathcal C$ is a contour that encircles the origin once in positive
direction. We choose $z$ and the
radius of this circle so that $1>|q|^{M}>z>0$.
(In particular, we choose $z$ to be real and positive.)
The sum over~$N$ can be evaluated by means of the $q$-binomial theorem
 (cf.\ \cite[Eq.~(1.3.2)]{GaRaAF})
\begin{equation} \label{eq:qbinth} 
\sum_{N\ge0}\frac {(\alpha;q)_N} {(q;q)_N}Z^N
 =  \frac {{( \alpha Z;q)_\infty}}{{(Z;q)_\infty}}
\end{equation}
with $\alpha=q^{2M+1}$ and $Z=z/q^M$.
Thereby we obtain
\begin{align*}
\sum_{N\ge0}p\left({MN} -A,2M,N\right)z^N
&=
\frac {1} {2\pi i}\int_{\mathcal C}\frac {dq} {q^{-A+1}}
\frac {(zq^{M+1};q)_{\infty}} {(zq^{-M};q)_{\infty}}\\
&=
\frac {1} {2\pi i}\int_{\mathcal C}\frac {dq} {q^{-A+1}}
\frac {1} {(zq^{-M};q)_{2M+1}}.
\end{align*}
At this point, we do a partial fraction expansion with respect to~$z$,
to see that\footnote{The denominator is square-free; therefore the
partial fraction expansion is easily obtained by making the Ansatz
$$
\frac {1} {(zq^{-M};q)_{2M+1}}
=\sum_{j=-M}^{M}\frac {a_j} {1-zq^j},
$$
and computing $a_j$ by multiplying both sides by $1-zq^j$
and subsequently setting $z=q^{-j}$.}
$$
\frac {1} {(zq^{-M};q)_{2M+1}}
=\sum_{j=-M}^{M}\frac {1} {1-zq^j}
\times\frac {1} {(q^{-M-j};q)_{M+j}\,(q;q)_{M-j}}.
$$
Upon substitution in the above integral, this shows that
\begin{align*}
\sum_{N\ge0}p\left({MN} -A,2M,N\right)&z^N\\
&\kern-4pt
=
\sum_{j=-M}^{M}
\frac {1} {2\pi i}\int_{\mathcal C}\frac {dq} {q^{-A+1}}
\frac 1 {1-zq^j}
\times\frac {(-1)^{M+j}q^{\binom {M+j+1}2}} {(q;q)_{M+j}
  \,(q;q)_{M-j}}\\
&\kern-4pt
=\sum_{j=-M}^{M}
\frac {1} {2\pi i}\int_{\mathcal C}\frac {dq} {q^{-A+1}}
\frac 1 {1-zq^j}
\times\frac {(-1)^{M+j}q^{\binom {M+j+1}2}} {(q;q)_{2M}}
\bmatrix 2M\\M-j\endbmatrix_{q}.
\end{align*}
The plan now is to apply the residue theorem to compute the
integral. Clearly, the singularities of the integrand are the
zeros of the denominator. For fixed~$j\ne0$, these are
$\om_{|j|}^\ell z^{-1/j}$, $\ell=0,1,\dots,|j|-1$, where $\om_{|j|}$ is a
primitive $|j|$-th root of unity, and several roots of unity resulting
from the factor $(q;q)_{2M}$.
Since we have chosen the contour~$\mathcal C$ so that $q$ satisfies $|q|<1$,
and since in the residue
theorem we only have to consider singularities inside of the
contour~$\mathcal C$, the roots of unity do not concern us.
Furthermore, by our assumption that $z<1$, for $j>0$ we have
$|\om_{|j|}^\ell z^{-1/j}|>1$, so that the corresponding term has no
singularities inside the contour~$\mathcal C$ and may therefore be ignored.
Finally, the term for $j=0$
has only singularities on the unit circle, and consequently it may also
be ignored.

As a result, the residue theorem yields
{\allowdisplaybreaks
\begin{align}
\notag
\sum_{N\ge0}&p\left({MN} -A,2M,N\right)z^N\\
\notag
&\kern-10pt
=
\sum_{j=-M}^{-1}
\sum_{\ell=0}^{|j|-1}
\frac 1
      {-jz\left(\om_{|j|}^\ell z^{-1/j}\right)^{j-1}}
  \times\frac
     {(-1)^{M+j}\left(\om_{|j|}^\ell z^{-1/j}\right)^{A-1+\binom {M+j+1}2}}
     {\left(\om_{|j|}^{\ell} z^{-1/j};
       \om_{|j|}^{\ell} z^{-1/j}\right)_{2M}}
 \bmatrix 2M\\M-j\endbmatrix_{\om_{|j|}^{\ell} z^{-1/j}}\\
\notag
&\kern-10pt
=
\sum_{j=1}^{M}
\sum_{\ell=0}^{j-1}
\frac {\om_{j}^\ell z^{1/j}}
      {j}
  \times\frac
     {(-1)^{M-j}\left(\om_{j}^\ell z^{1/j}\right)^{A-1+\binom {M-j+1}2}}
     {\left(\om_{j}^{\ell} z^{1/j};
       \om_{j}^{\ell} z^{1/j}\right)_{2M}}
 \bmatrix 2M\\M+j\endbmatrix_{\om_{j}^{\ell} z^{1/j}}\\
\notag
&\kern-10pt
=
\sum_{j=1}^{M}
\sum_{\ell=0}^{j-1}
\frac 1 {j}
  \times\frac
     {(-1)^{M-j}\left(\om_{j}^\ell z^{1/j}\right)^{A+\binom {M-j+1}2}}
     {\left(\om_{j}^{\ell} z^{1/j};
       \om_{j}^{\ell} z^{1/j}\right)_{2M-1}\,
       \left(1-(\om_{j}^{\ell} z^{1/j})^{\gcd(M-j,M+j)}\right)}\\
&\kern5cm
\times
\frac {1-(\om_{j}^{\ell} z^{1/j})^{\gcd(M-j,M+j)}}
      {1-({\om_{j}^\ell z^{1/j}})^{2M}}
\bmatrix 2M\\M+j\endbmatrix_{\om_{j}^{\ell} z^{1/j}}.
\label{eq:2}
\end{align}}%

We now need several auxiliary results.

\begin{lemma} \label{lem:3}
The term $1-(\om_{j}^{\ell} z^{1/j})^{\gcd(M-j,M+j)}$ divides
$1-z^2$ as a polynomial in~$z^{1/j}$.
\end{lemma}

\begin{proof}
We have $\gcd(M-j,M+j)=\gcd(M-j,2j)$. In particular, the number\break
$\gcd(M-j,M+j)$ divides~$2j$. Since
$(\om_{j}^{\ell} z^{1/j})^{2j}=z^2$, this immediately implies
the assertion of the lemma.
\end{proof}

\begin{lemma} \label{lem:4}
The term $\left(\om_{j}^{\ell} z^{1/j};
       \om_{j}^{\ell} z^{1/j}\right)_{2M-1}$ divides $(z;z)_{2M-1}$
as a polynomial in~$z^{1/j}$.
\end{lemma}

\begin{proof}
This follows by applying the argument of the proof of the previous
lemma to each factor $1-(\om_{j}^{\ell} z^{1/j})^r$, $r=1,2,\dots,2M-1$,
separately.
\end{proof}

\begin{lemma} \label{lem:5}
For positive integers $a$ and $b$,
the expression $\frac {1-Q^{\gcd(a,b)}} {1-Q^{a+b}}
\left[\smallmatrix a+b\\a\endsmallmatrix\right]_Q$ is a polynomial in~$Q$.
\end{lemma}

\begin{proof}
This is easy to show by counting cyclotomic polynomials as
factors in the numerator and denominator of the expression,
See e.g.\ \cite[Lemma~D.1]{KrStAA} (where it is proved in
addition that all coefficients are nonnegative).
\end{proof}

Now everything is in place for the proof of Theorem~\ref{thm:even}.

\begin{proof}[Proof of Theorem \ref{thm:even}]
By Lemmas~\ref{lem:3} and~\ref{lem:4} and Lemma~\ref{lem:5} with
$a=M+j$ and $b=M-j$, 
the denominator of the summand on the right-hand side
of~\eqref{eq:2} is
$$
\left(\om_{j}^{\ell} z^{1/j};
       \om_{j}^{\ell} z^{1/j}\right)_{2M-1}\,
       \left(1-(\om_{j}^{\ell} z^{1/j})^{\gcd(M-j,M+j)}\right),
$$
and it divides $(z;z)_{2M-1}\,(1-z^2)$ as a polynomial in~$z^{1/j}$.
On the other hand, as we have argued earlier, we know a priori that
our generating function of interest --- the left-hand side of~\eqref{eq:2} --- is a rational function in~$z$ (sic!), hence the right-hand side of~\eqref{eq:2} is as well. 
The conclusion is that there exist polynomials
$S(\,.\,)$ and $T(\,.\,)$ in
$\mathbb C[z]$, and a polynomial
$R(\,.\,)$ in $\mathbb C[z,z^{1/2},z^{1/3},\dots,z^{1/M}]$
such that our generating function can be written in the two forms 
\begin{equation} \label{eq:3} 
\frac {R(z,z^{1/2},\dots,z^{1/M})}
      {(z;z)_{2M-1}\,(1-z^2)}
=\frac {S(z)} {T(z)}.
\end{equation}
Rearranging terms, we infer
$$
{R(z,z^{1/2},\dots,z^{1/M})}
=\frac {S(z)\,{(z;z)_{2M-1}\,(1-z^2)}} {T(z)}.
$$
From the outset, this is an identity between {\it formal power series in $z,z^{1/2},z^{1/3},\dots,z^{1/M}$}.
However, on the left-hand side we find a {\it polynomial\/} in 
$z,z^{1/2},z^{1/3},\dots,z^{1/M}$, and on the right-hand side we find a {\it formal power series in~$z$}, Hence, 
${R(z,z^{1/2},\dots,z^{1/M})}$ 
must actually be a polynomial in~$z$.
By the left-hand side of~\eqref{eq:3}, this establishes the assertion
of Theorem~\ref{thm:even} about the denominator of the generating
function.

In order to establish also the assertion \eqref{eq:even} about the
numerator, we must look at the expression~\eqref{eq:2} in detail.
By comparing with the formula in Lemma~\ref{lem:s}, we realize that
the sum over~$\ell$ in~\eqref{eq:2} is the image of the Huffing
operator~$H_j$ when applied to the polynomial that arises from the
summand (without the prefactor~$\frac {1} {j}$) by replacing
$\om_j^\ell z^{1/j}$ by~$z$. 
This then leads to the expression
in~\eqref{eq:even} for the numerator $\text{Num}_e(M,r)$.

\medskip
The proof of Theorem~\ref{thm:even} is now complete.
\end{proof}

\medskip
Now let $m$ be odd, say $m=2M-1$. In this case, we have to consider
\begin{multline*}
p\left(\left\lfloor{\frac {(2M-1)N}2}\right\rfloor -A,2M-1,N\right)\\
=\begin{cases}
\coef{q^{(2M-1)\frac {N} {2}-A}}
\bmatrix 2M-1+N\\2M-1\endbmatrix_q,&\text{if $N$ is even},\\
\coef{q^{(2M-1)\frac {N} {2}-\frac {1} {2}-A}}
\bmatrix 2M-1+N\\2M-1\endbmatrix_q,&\text{if $N$ is odd}.
\end{cases}
\end{multline*}
We compute the generating function 
$$
\sum_{N\ge0}p\left(\left\lfloor\frac {(2M-1)N} {2}\right\rfloor-A,2M-1,N\right)z^N
$$
separately for even~$N$ and for odd~$N$.
Again, using arguments very similar to those in the case where $m$ is even,
one can show that in both cases one obtains rational functions with
the denominators having exclusively roots of unity as zeros.

The even part is
\begin{align*}
\sum_{k\ge0}p\left((2M-1)k-A,2M-1,2k\right)z^{2k}
&=
\sum_{k\ge0}z^{2k}\coef{q^{(2M-1)k-A}}
\bmatrix 2M-1+2k\\2M-1\endbmatrix_q\\
&=
\sum_{k\ge0}z^{2k}\frac {1} {2\pi i}\int_{\mathcal C}
\frac {dq} {q^{(2M-1)k-A+1}}
\frac {(q^{2M};q)_{2k}} {(q;q)_{2k}}.
\end{align*}
The sum over $k$ can again be evaluated by means of the $q$-binomial
theorem in~\eqref{eq:qbinth} (more precisely: by the bisection of the
$q$-binomial theorem). We get
\begin{align*}
\sum_{k\ge0}&p\left((2M-1)k-A,2M-1,2k\right)z^{2k}\\
&=
\frac {1} {2\pi i}\int_{\mathcal C}
\frac {dq} {2q^{-A+1}}
\left(
\frac {(zq^{M+\frac {1} {2}};q)_\infty} {(zq^{-M+\frac {1} {2}};q)_\infty}
+
\frac {(-zq^{M+\frac {1} {2}};q)_\infty} {(-zq^{-M+\frac {1} {2}};q)_\infty}
\right)\\
&=
\frac {1} {2\pi i}\int_{\mathcal C}
\frac {dq} {2q^{-A+1}}
\left(
\frac {1} {(zq^{-M+\frac {1} {2}};q)_{2M}}
+
\frac {1} {(-zq^{-M+\frac {1} {2}};q)_{2M}}
\right).
\end{align*}
Similarly, we have
\begin{align*}
\sum_{k\ge0}p&\left(\left\lfloor\frac {(2M-1)(2k+1)}2\right\rfloor-A,
2M-1,2k+1\right)z^{2k+1}\\
&=
\sum_{k\ge0}z^{2k+1}\coef{q^{\frac {1} {2}(2M-1)(2k+1)-\frac {1} {2}-A}}
\bmatrix 2M-1+2k+1\\2M-1\endbmatrix_q\\
&=
\sum_{k\ge0}z^{2k+1}\frac {1} {2\pi i}\int_{\mathcal C}
\frac {dq} {q^{\frac {1} {2}(2M-1)(2k+1)-\frac {1} {2}-A+1}}
\frac {(q^{2M};q)_{2k+1}} {(q;q)_{2k+1}}.
\end{align*}
By the $q$-binomial theorem, we obtain
\begin{align*}
\sum_{k\ge0}p&\left(\left\lfloor\frac {(2M-1)(2k+1)}2\right\rfloor-A,
2M-1,2k+1\right)z^{2k+1}\\
&=
\frac {1} {2\pi i}\int_{\mathcal C}
\frac {dq} {2q^{-A+\frac {1} {2}}}
\left(
\frac {(zq^{M+\frac {1} {2}};q)_\infty} {(zq^{-M+\frac {1} {2}};q)_\infty}
-
\frac {(-zq^{M+\frac {1} {2}};q)_\infty} {(-zq^{-M+\frac {1} {2}};q)_\infty}
\right)\\
&=
\frac {1} {2\pi i}\int_{\mathcal C}
\frac {dq} {2q^{-A+\frac {1} {2}}}
\left(
\frac {1} {(zq^{-M+\frac {1} {2}};q)_{2M}}
-
\frac {1} {(-zq^{-M+\frac {1} {2}};q)_{2M}}
\right).
\end{align*}
Putting both together, we get
\begin{multline*}
\sum_{k\ge0}p\left(\left\lfloor\frac {(2M-1)k}2\right\rfloor-A,
2M-1,k\right)z^{k}\\
=
\frac {1} {2\pi i}\int_{\mathcal C}
\frac {dq} {2q^{-A+1}}
\left(
\frac {1+q^{\frac12}} {(zq^{-M+\frac {1} {2}};q)_{2M}}
+
\frac {1-q^{\frac12}} {(-zq^{-M+\frac {1} {2}};q)_{2M}}
\right).
\end{multline*}
Next we do partial fraction decomposition with the denominators,
$$
\frac {1} {(\pm zq^{-M+\frac {1} {2}};q)_{2M}}
=\sum_{j=-M}^{M-1}\frac {1} {1\mp zq^{j+\frac {1} {2}}}
\times\frac {1} {(q^{-M-j};q)_{M+j}\,(q;q)_{M-j-1}}.
$$
Upon substitution in the above integral, this shows that
\begin{align*}
\sum_{k\ge0}&p\left(\left\lfloor\frac {(2M-1)k}2\right\rfloor-A,
2M-1,k\right)z^{k}\\
&
=
\sum_{j=-M}^{M-1}
\frac {1} {2\pi i}\int_{\mathcal C}\frac {dq} {2q^{-A+1}}
\left(\frac {1+q^{\frac {1} {2}}} {1-zq^{j+\frac {1} {2}}}
+\frac {1-q^{\frac {1} {2}}} {1+zq^{j+\frac {1} {2}}}\right)
\times\frac {(-1)^{M+j}q^{\binom {M+j+1}2}} {(q;q)_{M+j}
  \,(q;q)_{M-j-1}}\\
&
=\sum_{j=-M}^{M-1}
\frac {1} {2\pi i}\int_{\mathcal C}\frac {dq} {q^{-A+1}}
\frac {1+zq^{j+1}} {1-z^2q^{2j+1}}
\times\frac {(-1)^{M+j}q^{\binom {M+j+1}2}} {(q;q)_{2M-1}}
\bmatrix 2M-1\\M-j-1\endbmatrix_{q}.
\end{align*}
The residue theorem then yields
{\allowdisplaybreaks%
\begin{align}
\notag
\sum_{k\ge0}&p\left(\left\lfloor\frac {(2M-1)k}2\right\rfloor-A,
2M-1,k\right)z^{k}\\
&
=\sum_{j=-M}^{-1}\sum_{\ell=0} ^{|2j+1|-1}
\left(\om_{|2j+1|}^\ell z^{-2/(2j+1)}\right)^{A-1}
\frac {1+z\left(\om_{|2j+1|}^\ell z^{-2/(2j+1)}\right)^{j+1}}
      {-(2j+1)z^2\left(\om_{|2j+1|}^\ell z^{-2/(2j+1)}\right)^{2j}}\\
\notag
&\kern1cm
      \times\frac {(-1)^{M+j}\left(\om_{|2j+1|}^\ell z^{-2/(2j+1)}\right)
        ^{\binom {M+j+1}2}} {(\left(\om_{|2j+1|}^\ell
        z^{-2/(2j+1)}\right);
        \left(\om_{|2j+1|}^\ell z^{-2/(2j+1)}\right))_{2M-1}}
\bmatrix 2M-1\\M-j-1\endbmatrix_{\om_{|2j+1|}^\ell z^{-2/(2j+1)}}\\
\notag
&
=\sum_{j=1}^{M}\sum_{\ell=0} ^{2j-2}
\left(\om_{2j-1}^\ell z^{2/(2j-1)}\right)^{A-1}
\frac {1+z\left(\om_{2j-1}^\ell z^{2/(2j-1)}\right)^{-j+1}}
      {(2j-1)z^2\left(\om_{2j-1}^\ell z^{2/(2j-1)}\right)^{-2j}}\\
\notag
&\kern1cm
      \times\frac {(-1)^{M-j}\left(\om_{2j-1}^\ell z^{2/(2j-1)}\right)
        ^{\binom {M-j+1}2}} {\left(\om_{2j-1}^\ell
        z^{2/(2j-1)};
        \om_{2j-1}^\ell z^{2/(2j-1)}\right)_{2M-1}}
\bmatrix 2M-1\\M-j\endbmatrix_{\om_{2j-1}^\ell z^{2/(2j-1)}}\\
\notag
&
=\sum_{j=1}^{M}\sum_{\ell=0} ^{2j-2}
\left(\om_{2j-1}^\ell z^{2/(2j-1)}\right)^A
\frac {1}
      {2j-1}\\
\notag
&\kern1cm
  \times\frac {(-1)^{M-j}\left(\om_{2j-1}^\ell z^{2/(2j-1)}\right)
        ^{\binom {M-j+1}2}}
    {\left(1-\left(-\om_{4j-2}\right)^{\ell}
        z^{1/(2j-1)}\right)\,\left(\left(\om_{2j-1}^\ell
        z^{2/(2j-1)}\right)^2;
        \om_{2j-1}^\ell z^{2/(2j-1)}\right)_{2M-3}\,(1-z^2)}\\
&\kern1cm
\times
\frac {1-z^2} {1-\left(\om_{2j-1}^\ell z^{2/(2j-1)}\right)^{2M-1}}
\bmatrix 2M-1\\M-j\endbmatrix_{\om_{2j-1}^\ell z^{2/(2j-1)}}.
\label{eq:4}
\end{align}}%

Again, we need several auxiliary results.

\begin{lemma} \label{lem:6}
The term $1-(-\om_{4j-2})^{\ell} z^{1/(2j-1)}$ divides
$1-z$ as a polynomial in~$z^{1/(2j-1)}$.
\end{lemma}

\begin{proof}
We have
\begin{align*}
1-(-\om_{4j-2})^{\ell} z^{1/(2j-1)}
&=1-\left(e^{2\pi i(2j-1)/(4j-2)}e^{2\pi i/(4j-2)}\right)^{\ell} z^{1/(2j-1)}\\
&=1-\om_{2j-1}^{j\ell} z^{1/(2j-1)}.
\end{align*}
This is visibly a divisor of $1-z$.
\end{proof}

\begin{lemma} \label{lem:7}
The term $1-(\om_{2j-1}^{\ell} z^{2/(2j-1)})^{\gcd(M+j-1,M-j)}$ divides
$1-z^2$ as a polynomial in~$z^{2/(2j-1)}$.
\end{lemma}

\begin{proof}
We have $\gcd(M+j-1,M-j)=\gcd(M+j-1,2j-1)$. In particular, the number
$\gcd(M+j-1,M-j)$ divides~$2j-1$. Since
$(\om_{2j-1}^{\ell} z^{2/(2j-1)})^{2j-1}=z^2$, this immediately implies
the assertion of the lemma.
\end{proof}

\begin{lemma} \label{lem:8}
The term $\left(\left(\om_{2j-1}^\ell z^{2/(2j-1)}\right)^2;
\om_{2j-1}^\ell z^{2/(2j-1)}\right)_{2M-3}$
divides $(z^4;z^2)_{2M-3}$
as a polynomial in~$z^{2/(2j-1)}$.
\end{lemma}

\begin{proof}
This follows by applying the argument of the proof of 
Lemma~\ref{lem:3} to each factor
$1-(\om_{2j-1}^{\ell} z^{2/(2j-1)})^r$, $r=2,3,\dots,2M-2$,
separately.
\end{proof}

We are ready for the proof of Theorem~\ref{thm:odd}.

\begin{proof}[Proof of Theorem \ref{thm:odd}]
By Lemmas~\ref{lem:6}, \ref{lem:7}, \ref{lem:8}, and Lemma~\ref{lem:5} with $a=M+j-1$ and $b=M-j$,
the denominator of the summand on the right-hand side
of~\eqref{eq:4} is
$$
\left(1-\left(-\om_{4j-2}\right)^{\ell}
        z^{1/(2j-1)}\right)\,\left(\left(\om_{2j-1}^\ell
        z^{2/(2j-1)}\right)^2;
        \om_{2j-1}^\ell z^{2/(2j-1)}\right)_{2M-3}\,(1-z^2),
$$
and it divides $(1-z)\,(z^2;z^2)_{2M-2}$ as a polynomial in~$z^{1/(2j-1)}$.
Arguments analogous to the ones at the end of the proof 
of Theorem~\ref{thm:odd} then complete this proof.
Here, instead of images of~$H_j$, we deal with images of~$H_{2j-1}$.
One little detail is that one must observe that $-\om_{4j-2}$
is a primitive $(2j-1)$-th (!) root of unity, and that
$(-\om_{4j-2})^2=\om_{2j-1}$.
\end{proof}

\section{Acknowledgments}

The authors are grateful for the previous work of Arturo Martinez and
Angelica Castillo in computing the 144 constituents for $p(n,4,N)$.  
The authors would like to thank Dennis Eichhorn for his help revising
recent drafts of this work. We are furthermore indebted to
Kathrin Bringmann and Nicolas Smoot for the organization of a Section
on Number Theory within the program of the \"OMG--DMV-Congress 2025 at the
Johannes Kepler Universit\"at Linz, during which the second author
presented a preliminary report, which raised the interest of
the first author and led him to enter the project.

\setlength{\bibsep}{0pt} 
\renewcommand{\bibfont}{\small} 
\bibliographystyle{abbrv}
\bibliography{ref}

@article{KrStAA,
    author = {Christian Krattenthaler and Christian Stump},
    title = {Positive $m$-divisible non-crossing partitions and their {Kreweras} maps},
    journal = {\em Preprint, \tt ar$\chi$iv:2506.14996},
    year = 2025
}

@book{StanBT,
    author = {R. P. Stanley},
    title = {Enumerative
Combinatorics, vol.~1, second edition},
    publisher = {Cambridge
University Press},
    year = {Cambridge, 2012.}
}

@article{ProctorUni,
	Author = {R. A. Proctor},
	Journal = {Amer. Math. Monthly},
	Volume = {89},
	Pages = {721--734},
	Title = {Solution of two difficult combinatorial problems with linear algebra},
	Year = {1982}}

@inproceedings{ZeilbergerUni2,
	Author = {Doron Zeilberger},
	Title = {A one-line high school algebra proof of the unimodality of the {Gaussian} polynomials},
	BookTitle = {$q$-{Series} and {Partitions}},
	Pages = {35--44},
	Editor = {Stanton, Dennis},
	Publisher = {Springer Berlin},
	Series = {{IMA} {Volumes} in {Mathematics} and its {Applications}, vol.~18},
	Year = {1989}}

@article{StantonZeilbergerUni,
	Author = {Dennis Stanton and Doron Zeilberger},
	Journal = {Proc. Amer. Math. Soc.},
	Volume = {107},
	Pages = {39--42},
	Title = {O'{H}ara's unimodality proof and the {Odlyzko} Conjecture},
	Year = {1989}}

@article{GuptaTechnique,
	author = {J. Aniceto and B. Kronholm},
	journal = {Int. J. Number Theory},
	pages = {639--655},
	title = {{H}ansraj {G}upta's ``{A} Technique in Partitions" revisited: congruences, cranks, and polyhedral geometry},
	volume = {21},
	year = {2025}}

@article{gupta,
 ISSN = {05228441},
 URL = {http://www.jstor.org/stable/43667715},
 author = {Hansraj Gupta},
 journal = {Publikacije Elektrotehničkog fakulteta. Serija Matematika i fizika},
 volume = {498},
 pages = {73--76},
 publisher = {University of Belgrade, Serbia},
 title = {A TECHNIQUE IN PARTITIONS},
 year = {1975}
}

@article{BressoudUni,
	Author = {David M. Bressoud},
	Journal = {Discrete Math.},
	Volume = {99},
	Pages = {17--24},
	Title = {Unimodality of {Gaussian} polynomials},
	Year = {1992}}

@article{ZeilbergerOhara,
	Author = {Doron Zeilberger},
	Journal = {Amer. Math. Monthly},
	Pages = {590--602},
	Title = {Kathy {O}'{H}ara's constructive proof of the unimodality of the {G}aussian polynomials},
	Volume = {96},
	Year = {1989}}

@inproceedings{Uncu,
	Author = {Koutschan, C. and Uncu, A. K. and Wong, E.},
	Booktitle = {ISSAC 2023 -- Proceedings of the 2023 International Symposium on Symbolic and Algebraic Computation},
	Editor = {Gabriela Jeronimo},
	Pages = {434--442},
	Publisher = {Association for Computing Machinery},
	Title = {A Unified Approach to Unimodality of {G}aussian Polynomials},
	Year = {2025}}

@book{SchurUni,
	Address = {Berlin},
	Author = {I. J. Schur},
	Note = {edited by Helmut Grunsky},
	Publisher = {Springer-Verlag},
	Title = {Vorlesungen \"uber Invariantentheorie},
	Volume = {143},
	Year = {1968}}

@article{LSAMPpn3N,
	Author = {Angelica Castillo and Stephanie Flores and Anabel Hernandez and Brandt Kronholm and Acadia Larsen and Arturo Martinez},
	Journal = {Ann. Comb.},
	Number = {3-4},
	Pages = {589--611},
	Title = {Quasipolynomials and maximal coefficients of {G}aussian polynomials},
	Volume = {23},
	Year = {2019}}

@article{BurdeWagemann,  
     Author = {Burde, D. and  Wagemann, F.},
     journal = {\em Preprint, {\tt ar$\chi$iv:2411.14952}},
     year = {2024},
     url = {https://arxiv.org/abs/2411.14952v1},
     Title = {Cohomology of Perfect Lie Algebras}
     }

@article{KronholmINTEGERS,
	Author = {Kronholm, Brandt},
	Journal = {INTEGERS},
	Month = {March},
	Volume = {7},
	Pages = {Art.~\#A16, 1--6},
	Title = {On Congruence Properties of Consecutive Values of $p(n,m)$},
	Year = {2007}}

@article{BEK-6j-1,
	Author = {Breuer, Felix and Eichhorn, Dennis and Kronholm, Brandt},
	Journal = {Europ. J. Combin.},
	Month = {June},
	Pages = {230--252},
	Title = {Polyhedral geometry, supercranks, and combinatorial witnesses of congruence properties of partitions into three parts},
	Volume = {65},
	Year = {2017}}

@book{andrews1998theory,
	Author = {Andrews, George E.},
	Publisher = {Cambridge University Press},
	Title = {{The Theory of Partitions}},
	Year = {1998}}

@inproceedings{Sylvester2,
	Address = {Chelsea New York},
	Author = {James Joseph Sylvester},
	Booktitle = {Collected Math. Papers},
	Title = {Proof of the hitherto undemonstrated fundamental theorem of invariants},
	Volume = {3},
	pages = {117--126},
	Year = {1973}}

@article{OHara,
	Author = {Kathleen M. {O}{'}Hara},
	Journal = {J. Combin. Theory Ser.~A},
	Number = {3},
	Pages = {29--52},
	Title = {Unimodality of {Gaussian} coefficients: a constructive proof},
	Volume = {53},
	Year = {1990}}

@article{PakPanova,
	Author = {Igor Pak and Greta Panova},
	Journal = {J. Algebraic Combin.},
	Number = {4},
	Pages = {1103--1120},
	Title = {Unimodality via {K}ronecker products},
	Volume = {40},
	Year = {2014}}

@book{GaRaAF,
	Address = {Cambridge},
	Author = {Gasper, G. and Rahman, M.},
	Edition = {Second},
	Publisher = {Cambridge University Press},
	Series = {Encyclopedia of Mathematics and Its Applications},
	Title = {Basic {H}ypergeometric {S}eries},
	Volume = {96},
	Year = {2004}}

@article{Hahn2017pn3N,
	Author = {Heekyoung Hahn and Jisuh Huh and Eunsung Lim and Jaebum Sohn},
	Journal = {Ann. Comb.},
	Volume = {22},
	Pages = {543--562},
	Title = {From Partition Identities to a Combinatorial Approach to Explicit {S}atake Inversion},
	Year = {2018}}

@article{HahnPAMS2016,
	Author = {Heekyoung Hahn},
	Journal = {Proc. Amer. Math. Soc.},
	Number = {12},
	Pages = {5061--5069},
	Title = {ON TENSOR THIRD {$L$}-FUNCTIONS OF AUTOMORPHIC REPRESENTATIONS OF {$GL_n(\mathbb A_F)$}},
	Volume = {144},
	Year = {2016}}

@article{ehrhartpolynomial,
	Author = {Ehrhart, Eug{\`e}ne},
	Journal = {C. R. Acad. Sci. Paris},
	Mrclass = {10.25 (52.10)},
	Mrnumber = {0130860 (24 \#A714)},
	Pages = {616--618},
	Title = {Sur les poly\`edres rationnels homoth\'etiques \`a {$n$}\ dimensions},
	Volume = {254},
	Year = {1962}}

\appendix
\section*{Appendix}
\label{A5}

\setcounter{equation}{0}%
\global\def\thetheorem{\mbox{A.\arabic{theorem}}}
\global\def\theequation{\mbox{A.\arabic{equation}}}

In this appendix, we display the perpendicular generating functions
$$\sum_{N=0}^{\infty} p\left(\left\lfloor \frac{mN}{2}
\right\rfloor-A,m,N\right)z^{N}$$
for $m=5$ and $m=6$, continuing the identities for $m=1,2,3,4$
given in Propositions~\ref{p(n,1,N)}--\ref{C_4b}. As we said there,
these have been computed using our formulas in
Theorems~\ref{thm:even} and~\ref{thm:odd} using
{\sl Mathematica}, and they allow one to go much further.
The corresponding implementation is available
in the notebook file
{\tt orthview.nb} accompanying this article.

\begin{proposition}\label{GF5}
The partition generating function
$\sum_{N=0}^{\infty} p\left(\left\lfloor \frac{5N}{2}
  \right\rfloor-A,5,N\right)z^{N}$ has $15$~cases:
\begin{enumerate}
\item   \label{A5.0} $\displaystyle \sum_{N=0}^{\infty} p\left(\left\lfloor \frac{5N}{2} \right\rfloor-15a,5,N\right)z^{N}  =\newline z^{6 a}(-z^{4 a+1}-3 z^{4 a+2}-2 z^{4 a+3}-8 z^{4 a+4}-5z^{4 a+5}-12 z^{4 a+6}-6 z^{4 a+7}-13 z^{4 a+8}-7 z^{4 a+9}-10 z^{4 a+10}-4z^{4 a+11}-6 z^{4a+12}-2 z^{4 a+13}-2 z^{4 a+14}+z^{24 a+6}+z^{24 a+10}+z^{15}+2 z^{14}+3z^{13}+9 z^{12}+7 z^{11}+14 z^{10}+11 z^9+19 z^8+12 z^7+15 z^6+10 z^5+11 z^4+5   z^3+4 z^2+z+1)\times\frac{1}{(1-z) \left(1-z^2\right) \left(1-z^4\right)\left(1-z^6\right) \left(1-z^8\right)}.$

\item   \label{A5.1}$\displaystyle \sum_{N=0}^{\infty} p\left(\left\lfloor \frac{5N}{2} \right\rfloor-(15a+1),5,N\right)z^{N}  =\newline z^{6 a+1} (-2 z^{4 a+1}-2 z^{4 a+2}-6 z^{4 a+3}-4 z^{4  a+4}-10 z^{4 a+5}-7 z^{4 a+6}-13 z^{4 a+7}-6 z^{4 a+8}-12 z^{4 a+9}-5 z^{4a+10}-8 z^{4 a+11}-2 z^{4a+12}-3 z^{4 a+13}-z^{4 a+14}+z^{24 a+7}+z^{24 a+11}+z^{14}+4z^{13}+4 z^{12}+9 z^{11}+8 z^{10}+16 z^9+12 z^8+17 z^7+12 z^6+16 z^5+8 z^4+9 z^3+4 z^2+4 z+1)\times\frac{1}{(1-z) \left(1-z^2\right) \left(1-z^4\right)\left(1-z^6\right) \left(1-z^8\right)}.$

\item  $\displaystyle \sum_{N=0}^{\infty} p\left(\left\lfloor \frac{5N}{2} \right\rfloor-(15a+2),5,N\right)z^{N}  =\newline z^{6 a+1} (-z^{4 a+1}-z^{4 a+2}-4 z^{4 a+3}-4 z^{4 a+4}-9  z^{4 a+5}-5 z^{4 a+6}-13 z^{4 a+7}-7 z^{4 a+8}-13 z^{4 a+9}-5 z^{4 a+10}-9 z^{4  a+11}-4 z^{4a+12}-4 z^{4 a+13}-z^{4 a+14}-z^{4 a+15}+z^{24 a+9}+z^{24 a+13}+z^{15}+z^{14}+4  z^{13}+5 z^{12}+11 z^{11}+10 z^{10}+15 z^9+12 z^8+19 z^7+11 z^6+14 z^5+7 z^4+9 z^3+3 z^2+2 z+1)\times\frac{1}{(1-z) \left(1-z^2\right) \left(1-z^4\right)   \left(1-z^6\right) \left(1-z^8\right)}.$
  
\item  $\displaystyle \sum_{N=0}^{\infty} p\left(\left\lfloor \frac{5N}{2} \right\rfloor-(15a+3),5,N\right)z^{N}  =\newline z^{6 a+2} (-z^{4 a+1}-3 z^{4 a+2}-2 z^{4 a+3}-8 z^{4a+4}-5 z^{4 a+5}-12 z^{4 a+6}-6 z^{4 a+7}-13 z^{4 a+8}-7 z^{4 a+9}-10 z^{4a+10}-4 z^{4 a+11}-6 z^{4a+12}-2 z^{4 a+13}-2 z^{4 a+14}+z^{24 a+10}+z^{24 a+14}+z^{14}+2z^{13}+6 z^{12}+6 z^{11}+11 z^{10}+10 z^9+18 z^8+12 z^7+17 z^6+10 z^5+14 z^4+7z^3+6 z^2+3 z+2)\times\frac{1}{(1-z) \left(1-z^2\right) \left(1-z^4\right)   \left(1-z^6\right) \left(1-z^8\right)}.$
   
\item   $\displaystyle \sum_{N=0}^{\infty} p\left(\left\lfloor \frac{5N}{2} \right\rfloor-(15a+4),5,N\right)z^{N}  =\newline z^{6 a+2}(-2 z^{4 a+2}-2 z^{4 a+3}-6 z^{4 a+4}-4 z^{4   a+5}-10 z^{4 a+6}-7 z^{4 a+7}-13 z^{4 a+8}-6 z^{4 a+9}-12 z^{4 a+10}-5 z^{4   a+11}-8 z^{4a+12}-2 z^{4 a+13}-3 z^{4 a+14}-z^{4 a+15}+z^{24 a+12}+z^{24 a+16}+2 z^{14}+3   z^{13}+6 z^{12}+7 z^{11}+14 z^{10}+10 z^9+17 z^8+12 z^7+18 z^6+10 z^5+11 z^4+6   z^3+6 z^2+2 z+1)\times\frac{1}{(1-z) \left(1-z^2\right) \left(1-z^4\right)   \left(1-z^6\right) \left(1-z^8\right)}.$
   
\item $\displaystyle \sum_{N=0}^{\infty} p\left(\left\lfloor \frac{5N}{2} \right\rfloor-(15a+5),5,N\right)z^{N}  =\newline z^{6 a+2} (-z^{4a+16}-z^{4 a+2}-z^{4 a+3}-4 z^{4 a+4}-4   z^{4 a+5}-9 z^{4 a+6}-5 z^{4 a+7}-13 z^{4 a+8}-7 z^{4 a+9}-13 z^{4 a+10}-5 z^{4   a+11}-9 z^{4a+12}-4 z^{4 a+13}-4 z^{4 a+14}-z^{4 a+15}+z^{24 a+14}+z^{24 a+18}+z^{15}+2   z^{14}+3 z^{13}+9 z^{12}+7 z^{11}+14 z^{10}+11 z^9+19 z^8+12 z^7+15 z^6+10   z^5+11 z^4+5 z^3+4 z^2+z+1)\times\frac{1}{(1-z) \left(1-z^2\right) \left(1-z^4\right)   \left(1-z^6\right) \left(1-z^8\right)}.$

\item  $\displaystyle \sum_{N=0}^{\infty} p\left(\left\lfloor \frac{5N}{2} \right\rfloor-(15a+6),5,N\right)z^{N}  =\newline z^{6 a+3} (-z^{4 a+2}-3 z^{4 a+3}-2 z^{4 a+4}-8 z^{4a+5}-5 z^{4 a+6}-12 z^{4 a+7}-6 z^{4 a+8}-13 z^{4 a+9}-7 z^{4 a+10}-10 z^{4  a+11}-4 z^{4a+12}-6 z^{4 a+13}-2 z^{4 a+14}-2 z^{4 a+15}+z^{24 a+15}+z^{24 a+19}+z^{14}+4  z^{13}+4 z^{12}+9 z^{11}+8 z^{10}+16 z^9+12 z^8+17 z^7+12 z^6+16 z^5+8 z^4+9  z^3+4 z^2+4 z+1)\times\frac{1}{(1-z) \left(1-z^2\right) \left(1-z^4\right)   \left(1-z^6\right) \left(1-z^8\right)}.$
   
\item  $\displaystyle \sum_{N=0}^{\infty} p\left(\left\lfloor \frac{5N}{2} \right\rfloor-(15a+7),5,N\right)z^{N}  =\newline z^{6 a+3} (-2 z^{4 a+3}-2 z^{4 a+4}-6 z^{4   a+5}-4 z^{4 a+6}-10 z^{4 a+7}-7 z^{4 a+8}-13 z^{4 a+9}-6 z^{4 a+10}-12 z^{4   a+11}-5 z^{4a+12}-8 z^{4 a+13}-2 z^{4 a+14}-3 z^{4 a+15}-z^{4a+16}+z^{24 a+17}+z^{24   a+21}+z^{15}+z^{14}+4 z^{13}+5 z^{12}+11 z^{11}+10 z^{10}+15 z^9+12 z^8+19 z^7+11 z^6+14 z^5+7 z^4+9 z^3+3 z^2+2 z+1)\times\frac{1}{(1-z) \left(1-z^2\right)   \left(1-z^4\right) \left(1-z^6\right) \left(1-z^8\right)}.$
   
\item  $\displaystyle \sum_{N=0}^{\infty} p\left(\left\lfloor \frac{5N}{2} \right\rfloor-(15a+8),5,N\right)z^{N}  =\newline z^{6 a+4} (-z^{4 a+2}-z^{4 a+3}-4 z^{4 a+4}-4  z^{4 a+5}-9 z^{4 a+6}-5 z^{4 a+7}-13 z^{4 a+8}-7 z^{4 a+9}-13 z^{4 a+10}-5 z^{4  a+11}-9 z^{4a+12}-4 z^{4 a+13}-4 z^{4 a+14}-z^{4 a+15}-z^{4a+16}+z^{24 a+18}+z^{24 a+22}+z^{14}+2  z^{13}+6 z^{12}+6 z^{11}+11 z^{10}+10 z^9+18 z^8+12 z^7+17 z^6+10 z^5+14 z^4+7  z^3+6 z^2+3 z+2)\times\frac{1}{(1-z) \left(1-z^2\right) \left(1-z^4\right)  \left(1-z^6\right) \left(1-z^8\right)}.$
  
\item   $\displaystyle \sum_{N=0}^{\infty} p\left(\left\lfloor \frac{5N}{2} \right\rfloor-(15a+9),5,N\right)z^{N}  =\newline z^{6 a+4} (-z^{4 a+3}-3 z^{4  a+4}-2 z^{4 a+5}-8 z^{4 a+6}-5 z^{4 a+7}-12 z^{4 a+8}-6 z^{4 a+9}-13 z^{4   a+10}-7 z^{4 a+11}-10 z^{4a+12}-4 z^{4 a+13}-6 z^{4 a+14}-2 z^{4 a+15}-2 z^{4a+16}+z^{24 a+20}+2   z^{14}+3 z^{13}+6 z^{12}+7 z^{11}+14 z^{10}+10 z^9+17 z^8+12 z^7+18 z^6+10   z^5+11 z^4+6 z^3+6 z^2+2 z+1)\times\frac{1}{(1-z) \left(1-z^2\right)  \left(1-z^4\right) \left(1-z^6\right) \left(1-z^8\right)}.$
   
\item $\displaystyle \sum_{N=0}^{\infty} p\left(\left\lfloor \frac{5N}{2} \right\rfloor-(15a+10),5,N\right)z^{N}  =\newline z^{6 a+4} (-2 z^{4 a+4}-2 z^{4 a+5}-6 z^{4   a+6}-4 z^{4 a+7}-10 z^{4 a+8}-7 z^{4 a+9}-13 z^{4 a+10}-6 z^{4 a+11}-12 z^{4a+12}-5 z^{4  a+13}-8 z^{4 a+14}-2 z^{4 a+15}-3 z^{4a+16}-z^{4 a+17}+z^{24 a+22}+z^{24 a+26}+z^{15}+2 z^{14}+3 z^{13}+9 z^{12}+7 z^{11}+14 z^{10}+11 z^9+19 z^8+12 z^7+15 z^6+10 z^5+11 z^4+5 z^3+4 z^2+z+1)\times\frac{1}{(1-z) \left(1-z^2\right) \left(1-z^4\right)   \left(1-z^6\right) \left(1-z^8\right)}.$

\item $\displaystyle \sum_{N=0}^{\infty} p\left(\left\lfloor \frac{5N}{2} \right\rfloor-(15a+11),5,N\right)z^{N}  =\newline z^{6 a+5} (-z^{4 a+3}-z^{4 a+4}-4 z^{4 a+5}-4  z^{4 a+6}-9 z^{4 a+7}-5 z^{4 a+8}-13 z^{4 a+9}-7 z^{4 a+10}-13 z^{4 a+11}-5 z^{4a+12}-9 z^{4 a+13}-4 z^{4 a+14}-4 z^{4 a+15}-z^{4a+16}-z^{4 a+17}+z^{24 a+23}+z^{24 a+27}+z^{14}+4 z^{13}+4 z^{12}+9 z^{11}+8 z^{10}+16 z^9+12 z^8+17 z^7+12 z^6+16z^5+8 z^4+9 z^3+4 z^2+4 z+1)\times\frac{1}{(1-z) \left(1-z^2\right) \left(1-z^4\right)   \left(1-z^6\right) \left(1-z^8\right)}.$
   
\item $\displaystyle \sum_{N=0}^{\infty} p\left(\left\lfloor \frac{5N}{2} \right\rfloor-(15a+12),5,N\right)z^{N}  =\newline z^{6 a+5} (-z^{4 a+4}-3 z^{4 a+5}-2 z^{4   a+6}-8 z^{4 a+7}-5 z^{4 a+8}-12 z^{4 a+9}-6 z^{4 a+10}-13 z^{4 a+11}-7 z^{4a+12}-10 z^{4   a+13}-4 z^{4 a+14}-6 z^{4 a+15}-2 z^{4a+16}-2 z^{4 a+17}+z^{24 a+25}+z^{24   a+29}+z^{15}+z^{14}+4 z^{13}+5 z^{12}+11 z^{11}+10 z^{10}+15 z^9+12 z^8+19   z^7+11 z^6+14 z^5+7 z^4+9 z^3+3 z^2+2 z+1)\times\frac{1}{(1-z) \left(1-z^2\right)   \left(1-z^4\right) \left(1-z^6\right) \left(1-z^8\right)}.$
   
\item $\displaystyle \sum_{N=0}^{\infty} p\left(\left\lfloor \frac{5N}{2} \right\rfloor-(15a+13),5,N\right)z^{N}  =\newline z^{6 a+6} (-2 z^{4 a+4}-2 z^{4 a+5}-6 z^{4   a+6}-4 z^{4 a+7}-10 z^{4 a+8}-7 z^{4 a+9}-13 z^{4 a+10}-6 z^{4 a+11}-12 z^{4a+12}-5 z^{4   a+13}-8 z^{4 a+14}-2 z^{4 a+15}-3 z^{4a+16}-z^{4 a+17}+z^{24 a+26}+z^{24 a+30}+z^{14}+2   z^{13}+6 z^{12}+6 z^{11}+11 z^{10}+10 z^9+18 z^8+12 z^7+17 z^6+10 z^5+14 z^4+7   z^3+6 z^2+3 z+2)\times\frac{1}{(1-z) \left(1-z^2\right) \left(1-z^4\right)   \left(1-z^6\right) \left(1-z^8\right)}.$
   
\item $\displaystyle \sum_{N=0}^{\infty} p\left(\left\lfloor \frac{5N}{2} \right\rfloor-(15a+14),5,N\right)z^{N}  =\newline z^{6 a+6} (-z^{4 a+4}-z^{4 a+5}-4 z^{4 a+6}-4 z^{4 a+7}-9 z^{4 a+8}-5 z^{4 a+9}-13 z^{4 a+10}-7 z^{4 a+11}-13 z^{4a+12}-5 z^{4   a+13}-9 z^{4 a+14}-4 z^{4 a+15}-4 z^{4a+16}-z^{4 a+17}-z^{4 a+18}+z^{24 a+28}+z^{24 a+32}+2   z^{14}+3 z^{13}+6 z^{12}+7 z^{11}+14 z^{10}+10 z^9+17 z^8+12 z^7+18 z^6+10   z^5+11 z^4+6 z^3+6 z^2+2 z+1)\times\frac{1}{(1-z) \left(1-z^2\right)   \left(1-z^4\right) \left(1-z^6\right) \left(1-z^8\right)}.$
   \end{enumerate}
\end{proposition}

\begin{proposition}
The partition generating function
$\sum_{N=0}^{\infty} p\left(3N-A,6,N\right)z^{N}$ has six~cases:
\begin{enumerate} 
\item $\displaystyle \sum_{N=0}^{\infty} p\left(3N-6a,6,N\right)z^{N}  = \newline (z^{2 a}+2 z^{2 a+1}+6 z^{2 a+2}+10 z^{2 a+3}+14 z^{2 a+4}+15 z^{2 a+5}+14 z^{2 a+6}+10 z^{2 a+7}+6 z^{2 a+8}+2 z^{2 a+9}+z^{2 a+10}-2 z^{3a+1}-5 z^{3 a+2}-8 z^{3 a+3}-11 z^{3 a+4}-12 z^{3 a+5}-11 z^{3 a+6}-8 z^{3 a+7}-5 z^{3 a+8}-2 z^{3 a+9}+z^{6 a+3}+z^{6 a+4}+z^{6 a+5}+z^{6a+6}+z^{6 a+7})\newline\times \frac{1}{(1-z) \left(1-z^2\right)^2 \left(1-z^3\right) \left(1-z^4\right) \left(1-z^5\right)}.$

\item $\displaystyle \sum_{N=0}^{\infty} p\left(3N-(1+6a),6,N\right)z^{N}  = \newline (2 z^{2 a+1}+5 z^{2 a+2}+9 z^{2 a+3}+12 z^{2 a+4}+15 z^{2 a+5}+14 z^{2 a+6}+12 z^{2 a+7}+7 z^{2 a+8}+4 z^{2 a+9}+z^{2 a+10}-z^{3 a+1}-3z^{3 a+2}-6 z^{3 a+3}-10 z^{3 a+4}-12 z^{3 a+5}-12 z^{3 a+6}-10 z^{3 a+7}-6 z^{3 a+8}-3 z^{3 a+9}-z^{3 a+10}+z^{6 a+4}+z^{6 a+5}+z^{6a+6}+z^{6 a+7}+z^{6 a+8})\newline\times \frac{1}{(1-z) \left(1-z^2\right)^2 \left(1-z^3\right) \left(1-z^4\right) \left(1-z^5\right)}.$
   
\item $\displaystyle \sum_{N=0}^{\infty} p\left(3N-(2+6a),6,N\right)z^{N}  = \newline (z^{2 a+1}+4 z^{2 a+2}+7 z^{2 a+3}+12 z^{2 a+4}+14 z^{2 a+5}+15 z^{2 a+6}+12 z^{2 a+7}+9 z^{2 a+8}+5 z^{2 a+9}+2 z^{2 a+10}-2 z^{3 a+2}-5z^{3 a+3}-8 z^{3 a+4}-11 z^{3 a+5}-12 z^{3 a+6}-11 z^{3 a+7}-8 z^{3 a+8}-5 z^{3 a+9}-2 z^{3 a+10}+z^{6 a+5}+z^{6 a+6}+z^{6 a+7}+z^{6a+8}+z^{6 a+9})\newline\times \frac{1}{(1-z) \left(1-z^2\right)^2 \left(1-z^3\right) \left(1-z^4\right) \left(1-z^5\right)}.$
  
\item $\displaystyle \sum_{N=0}^{\infty} p\left(3N-(3+6a),6,N\right)z^{N}  = \newline (z^{2 a+1}+2 z^{2 a+2}+6 z^{2 a+3}+10 z^{2 a+4}+14 z^{2 a+5}+15 z^{2 a+6}+14 z^{2 a+7}+10 z^{2 a+8}+6 z^{2 a+9}+2 z^{2 a+10}+z^{2a+11}-z^{3 a+2}-3 z^{3 a+3}-6 z^{3 a+4}-10 z^{3 a+5}-12 z^{3 a+6}-12 z^{3 a+7}-10 z^{3 a+8}-6 z^{3 a+9}-3 z^{3 a+10}-z^{3 a+11}+z^{6a+6}+z^{6 a+7}+z^{6 a+8}+z^{6 a+9}+z^{6 a+10})\newline\times \frac{1}{(1-z) \left(1-z^2\right)^2 \left(1-z^3\right) \left(1-z^4\right) \left(1-z^5\right)}.$
   
\item $\displaystyle \sum_{N=0}^{\infty} p\left(3N-(4+6a),6,N\right)z^{N}  = \newline (2 z^{2 a+2}+5 z^{2 a+3}+9 z^{2 a+4}+12 z^{2 a+5}+15 z^{2 a+6}+14 z^{2 a+7}+12 z^{2 a+8}+7 z^{2 a+9}+4 z^{2 a+10}+z^{2 a+11}-2 z^{3 a+3}-5z^{3 a+4}-8 z^{3 a+5}-11 z^{3 a+6}-12 z^{3 a+7}-11 z^{3 a+8}-8 z^{3 a+9}-5 z^{3 a+10}-2 z^{3 a+11}+z^{6 a+7}+z^{6 a+8}+z^{6 a+9}+z^{6a+10}+z^{6 a+11})\newline\times \frac{1}{(1-z) \left(1-z^2\right)^2 \left(1-z^3\right) \left(1-z^4\right) \left(1-z^5\right)}.$

\item $\displaystyle \sum_{N=0}^{\infty} p\left(3N-(5+6a),6,N\right)z^{N}  = \newline (z^{2 a+2}+4 z^{2 a+3}+7 z^{2 a+4}+12 z^{2 a+5}+14 z^{2 a+6}+15 z^{2 a+7}+12 z^{2 a+8}+9 z^{2 a+9}+5 z^{2 a+10}+2 z^{2 a+11}-z^{3 a+3}-3z^{3 a+4}-6 z^{3 a+5}-10 z^{3 a+6}-12 z^{3 a+7}-12 z^{3 a+8}-10 z^{3 a+9}-6 z^{3 a+10}-3 z^{3 a+11}-z^{3 a+12}+z^{6 a+8}+z^{6 a+9}+z^{6a+10}+z^{6 a+11}+z^{6 a+12})\newline\times \frac{1}{(1-z) \left(1-z^2\right)^2 \left(1-z^3\right) \left(1-z^4\right) \left(1-z^5\right)}.$
\end{enumerate}
\end{proposition}

\end{document}